\documentclass[english]{amsart}
\usepackage[T1]{fontenc}
\usepackage[latin9]{inputenc}
\usepackage{geometry}
\usepackage{textcomp}
\usepackage{amsthm}
\usepackage{amssymb}
\usepackage{esint}
\usepackage[all]{xy}

\makeatletter
\numberwithin{equation}{section}
\numberwithin{figure}{section}
  \theoremstyle{plain}
  \newtheorem*{thm*}{\protect\theoremname}
\theoremstyle{plain}
\newtheorem{thm}{\protect\theoremname}
  \theoremstyle{remark}
  \newtheorem{rem}[thm]{\protect\remarkname}
  \theoremstyle{plain}
  \newtheorem{prop}[thm]{\protect\propositionname}
  \theoremstyle{plain}
  \newtheorem{cor}[thm]{\protect\corollaryname}
  \theoremstyle{definition}
  \newtheorem{defn}[thm]{\protect\definitionname}
  \theoremstyle{plain}
  \newtheorem{lem}[thm]{\protect\lemmaname}

\makeatother

\usepackage{babel}
  \providecommand{\corollaryname}{Corollary}
  \providecommand{\definitionname}{Definition}
  \providecommand{\lemmaname}{Lemma}
  \providecommand{\propositionname}{Proposition}
  \providecommand{\remarkname}{Remark}
  \providecommand{\theoremname}{Theorem}
\providecommand{\theoremname}{Theorem}

\begin{document}

\title{Classicality for small slope overconvergent automorphic forms on
some compact PEL Shimura varieties of type C}

\author{Christian Johansson}
\address{Department of Mathematics, Imperial College London, London SW7 2AZ, UK}
\email{h.johansson09@imperial.ac.uk}

\begin{abstract}
We study the rigid cohomology of the ordinary locus in some compact
PEL Shimura varieties of type C with values in automorphic local systems
and use it to prove a small slope criterion for classicality of overconvergent
Hecke eigenforms. This generalises the work of Coleman, and is a first step in an ongoing project to extend the cohomological approach to classicality to higher-dimensional Shimura varieties.

\end{abstract}

\maketitle

\section{Introduction}

A celebrated theorem of Coleman states that if $f$ is an overconvergent
modular form of weight $k\geq 2$ and tame level $\Gamma_{1}(N)$ which is an eigenform for $U_{p}$ with
slope (i.e. the $p$-adic valuation of the eigenvalue) less than $k-1$,
then $f$ is in fact a (classical) modular form of weight $k$ for
the congruence subgroup $\Gamma_{1}(N)\cap\Gamma_{0}(p)$. This theorem,
usually referred to either as a ``classicality theorem'' or ``control
theorem'', generalized a previous result of Hida for ordinary $p$-adic
modular forms. It is the key result needed for extending constructions
on classical modular forms (such as construction of Galois representations)
to overconvergent modular forms of finite slope by $p$-adic interpolation
since it implies that classical forms are dense in Coleman families
and on the Coleman-Mazur eigencurve. The Galois representations associated with finite slope
overconvergent modular eigenforms were investigated by Kisin in \cite{Kis}, and
it was shown that these Galois representations are trianguline at $p$ and satisfy the Fontaine-Mazur
conjecture. 

In attempting to generalize Coleman's geometric theory for $p$-adic
interpolation of modular forms to other PEL Shimura varieties one
quickly runs into two major obstacles; defining families and proving
the analogue of the classicality criterion. Both problems seem hard,
as Coleman's methods do not generalize in an obvious way (using methods
similar to those of Coleman, Kisin and Lai constructed one-dimensional
families of Hilbert modular forms; this has recently been extended
to the Siegel-Hilbert case by Mok and Tan \cite{MoTa}). Instead other
methods of $p$-adic interpolation were developed (see e.g. \cite{Buzz},
\cite{Che}, \cite{Loe}, \cite{Eme} and \cite{Urb}), which have
been applied with great success to the deformation theory of Galois
representations.

Recently there has been much progress on also in the geometric theory,
using methods that are very different to Coleman's; see \cite{AIP}
for the construction of families and \cite{PS1}, \cite{PS2} and
\cite{Tian} for classicality results. The method for proving classicality
originates from work of Kassaei \cite{Kas2}, building on previous
work by Buzzard and Taylor \cite{BuTa} on the strong Artin
conjecture for two-dimensional representations of $Gal(\overline{\mathbb{Q}}/\mathbb{Q})$, and is in essence a geometric way of analytically continuing
the overconvergent form to the whole modular curve (or more generally
Shimura variety) of Iwahori level at $p$. In particular it is entirely
different from Coleman's proof, which is cohomological in nature,
and instead requires a very explicit understanding of the geometry
of the Shimura variety and the geometry of the $U_{p}$-correspondence.

In this paper, we revisit Coleman's original method and generalize
it to certain compact PEL Shimura varieties of type $C$, which are
closely related to Hilbert modular varieties. For the exact definitions
of objects and results mentioned in this introduction we refer to
the main body of the text. To define our Shimura varieties, we start
with a quaternion division algebra $B$ over a totally real field
$F$ of degree $d$ over $\mathbb{Q}$. We fix a rational prime $p$
and assume that $B$ is split at all places above $p$ and also split at every real place of $F$. Such a $B$ then gives rise to
a PEL data in a standard way, hence a reductive group $G^{\star}$
over $\mathbb{Q}$ and given an open compact subgroup $K\subseteq G^{\star}(\mathbb{A}^{\infty})$
we get an associated Shimura variety. For a special choice of $K=K_{1}(c,N)$,
let us denote the corresponding Shimura variety by $X$. It has potentially good
reduction at $p$ and we may study the ordinary locus $X_{\mathbb{F}_{p}}^{ord}$
in characteristic $p$ and its lift $X_{rig}^{ord}$ inside the rigid
analytification of the generic fibre $X_{rig}$ of $X$. We may define and study spaces of classical
(resp. overconvergent) automorphic forms on $X_{rig}$ and $X_{rig}^{ord}$,
defined as sections (resp. sections overconvergent along the non-ordinary locus) of the appropriate sheaf (see
section \ref{sub:2.1 Aut vec bundle}). These spaces carry actions of appropriately defined
Hecke algebras, analogous to the situation for modular curves and Hilbert modular varieties. Decompose $p$ as $p=\mathfrak{p}_{1}^{e_{1}}...\mathfrak{p}_{r}^{e_{r}}$ in $F$ and let $d_{i}=[F_{\mathfrak{p}_{i}}:\mathbb{Q}_{p}]$. By fixing embeddings $\overline{\mathbb{Q}}\hookrightarrow\mathbb{C}$ and  $\overline{\mathbb{Q}}\hookrightarrow\overline{\mathbb{Q}}_{p}$ we may index the weights of our automorphic forms by the embeddings $F\hookrightarrow\overline{\mathbb{Q}}_{p}$; we label them $k_{ij}$ with  $1\leq i\leq r$, $1\leq j \leq d_{r}$. We define a quantity $\lambda (k_{1},...,k_{r})$ for integers $k_{1},...,k_{r}$ with $k_{i}\geq 2d_{i}$ by 
\[
\lambda (k_{1},...,k_{r})={\rm{inf}}_{i}\left((k_{i}-2d_{i}){\rm{inf}}(1/2,1/d_{i})\right)
\]
Our main theorem is the following: 

\begin{thm*}
1) (Theorem \ref{thm:classicality}(b) ) Let $f$ be an overconvergent
Hecke eigenform of weight $(k_{11},...,k_{rd_{r}})$ ($k_{ij}\geq2$ for
all $i,j$) with $U_{p}$-slope less than $\inf\left(k_{ij}-1,\lambda(k_{1},...,k_{r})\right)$
(here $p$ has valuation $1$ and $k_{i}=\sum_{j}k_{ij}$). Then its system of Hecke eigenvalues
comes from the $p$-stabilization of a classical form of level $K$.

2) (Theorem \ref{thm:classicality, d=00003D1}(b) ) Let $F=\mathbb{Q}$. Assume that $f$
is an overconvergent Hecke eigenform of weight $k\geq2$ with $U_{p}$-slope
less than $k-1$. Then its system of Hecke eigenvalues is classical
of level $\Gamma_{1}(N)\cap\Gamma_{0}(pq_{1}...q_{r})$ (where the
$q_{i}\ne p$ are the primes where $B$ is ramified). 
\end{thm*}

Let us briefly outline the contents of the paper. Section \ref{sec:1}
is devoted to setting up the basic definitions of $B$, $G^{\star}$
and the Shimura varieties involved. We recall two different integral models (due to Deligne and Pappas \cite{DePa} resp. Sasaki \cite{Sas}, the latter using ideas of Pappas and Rapoport on local models) and the algebraic representation theory of $G^{\star}$. In section \ref{sec:2}
we define $p$-adic and overconvergent automorphic forms on $X$ using
the automorphic vector bundles of Harris and Milne and define the
Hecke operators acting on them. We give two definitions in particular
of the $U_{p}$-operator and show that they agree. As in the theory
for modular curves one of the definitions uses the canonical subgroup
and therefore establishes a very direct link to the Frobenius morphism
in characteristic $p$. A key construction in Coleman's proof is that
of a sheaf homomorphism 
\[
\theta=\theta^{k-1}\,:\,\omega^{2-k}\rightarrow\omega^{k}
\]

As is no doubt well-known to experts, this is Faltings's BGG complex \cite{Fal}
for the modular curve (and weight $k$). In sections \ref{sub:2.4 BGG} and \ref{sub:2.5 FBGG}
we give the analogues on $X$. In particular, this gives a ``theta
map'' 
\[
\theta\,:\,\bigoplus_{i,j}H^{0}\left(X_{rig},W^{\dagger}\left(k_{11},...,2-k_{ij},...,k_{rd_{r}}\right)\right)\,\longrightarrow H^{0}\left(X_{rig},W^{\dagger}\left(k_{11},...,k_{rd_{r}}\right)\right)
\]

for weights with $k_{ij}\geq2$ for all $i$ and $j$; here $H^{0}\left(X_{rig},W^{\dagger}\left(k_{11}^{\prime},...,k_{rd_{r}}^{\prime}\right)\right)$
denotes the spaces of overconvergent automorphic forms of weight $(k_{11}^{\prime},...,k_{rd_{r}}^{\prime})$.

Section \ref{sec:3} is the main part of the paper. We begin by recalling
some notions from rigid cohomology and overconvergent de Rham cohomology,
and define certain overconvergent $F$-isocrystals $\mathcal{E}_{\underline{k}}$ that play a key role in the arguments, analogous to the sheaves $\mathcal{H}_{k}$ defined in \S 2 of \cite{Col}.
In section \ref{sub: 3.1 comparison} we prove the main comparison
theorem, analogous to Theorem 5.4 of \cite{Col}. It identifies, in
particular, the cokernel of $\theta$ with the degree $d$ rigid cohomology
of $\mathcal{E}_{\underline{k}}$ on $X^{ord}_{\overline{\mathbb{F}}_{p}}$, via Faltings's BGG complex. Section
\ref{sub: 3.2 small slope in coh} proves the analogue of the crucial
but innocent-looking Lemma 6.2 of \textit{op. cit.}, showing that
forms of slope less than ${\rm inf}\, k_{ij}-1$ are not in image of
$\theta$ and hence that their system of Hecke eigenvalues occur in
the cohomology of $\mathcal{E}_{\underline{k}}$. 

So far the arguments have made no essential use of any specific properties
of our Shimura varieties; indeed the results and proofs would carry
over for example to any compact PEL Shimura variety with nonvanishing
Hasse invariant, or any PEL Shimura curve with nonvanishing Hasse
invariant. In section \ref{sub:3.3 excision} we use the excision
sequence to reduce the understanding of the degree $d$ rigid cohomology
of $\mathcal{E}_{\underline{k}}$ on $X^{ord}_{\overline{\mathbb{F}}_{p}}$ to understanding the degree $d$
cohomology on $X_{\overline{\mathbb{F}}_{p}}^{PR}$ and the degree $d+1$ local cohomology on the complement. Here $^{PR}$ denotes that we are using the model of Sasaki (the "Pappas-Rapoport" model).
The former is well understood, using comparison theorems between various
cohomology theories, by the classical theory of automorphic forms
(Matsushima's formula). We remark that this is where it is necessary to use the Pappas-Rapoport model; the rigid cohomology of the singular special fiber of the Deligne-Pappas model will most likely not agree with the de Rham cohomology of the generic fiber. To understand the local cohomology group we use information about the slopes of nonordinary abelian varieties for our moduli problem and some results of Kedlaya \cite{Ked2} to
prove bounds for the Frobenius-slopes. The next section then
translates these bounds into information about the $U_{p}$-operator,
using the link between $U_{p}$ and Frobenius given by the canonical
subgroup, and deduces part 1) our main theorem above. Finally, for
completeness, the last section gives a different treatment of the case $F=\mathbb{Q}$ using
a (somewhat simplified) version of Coleman's dimension-counting argument,
establishing part 2) of the main theorem (which is stronger than the special case $F=\mathbb{Q}$ of part 1) ).

Let us make some remarks regarding our results. First of all, what
we prove is that certain systems of Hecke eigenvalues are classical,
rather than the stronger fact that the forms themselves are classical.
This is the price we pay for working with Hecke modules and the flexibility
they offer. If one had some control on the dimension of the Hecke
modules we work with (as Coleman has in \cite{Col}) or knew multiplicity
one for overconvergent automorphic forms one could hope to recover the classicality
of the forms themselves, but these results are not available in our
setting (except when $F=\mathbb{Q}$ where the first technique is
available to us, see Remark \ref{rem: strengthening}). However, for
applications to eigenvarieties and Galois representations this weakening
is unimportant, as one passes directly to systems of Hecke eigenvalues
anyway. As for optimality, the results of the paper are in general far from what
is expected. On the automorphic side one would conjecture (by comparison with the theory for groups compact at infinity \cite{Loe}) that an overconvergent eigenform
of slope less than $\inf\left(k_{ij}-1\right)$ has a classical system
of Hecke eigenvalues. Our theorem proves this for example when there is only one prime above $p$ (i.e. $r=1$), $d\geq 2$ and the weight is "not too parallel" (more precisely, under the condition that $\inf\left(k_{ij}\right)\leq \left(\sum k_{1j}/d\right)-1$; note that $\inf\left(k_{ij}\right)\leq \sum k_{1j}/d$ always holds). However, as a vague rule, the bound gets worse as $r$ gets bigger. This may be compared with the bounds obtained in \cite{PS1} in the unramified Hilbert setting, which are more uniform though not quite optimal. We should also mention, and are grateful to the referee for pointing out to us, that to state an optimal conjecture, one should look at the Galois side. Specifically, one should look at when trianguline representations are de Rham, which has been done by Nakamura \cite{Nak}. One could also conjecture that any overconvergent eigenform not in the image
of $\theta$ is classical (for modular curves this is Corollary 7.2.1
of \cite{Col}). We prove this in our case when $F=\mathbb{Q}$ and
obtain a partial result in this direction (Theorem \ref{thm:classicality}(a)
) when $F\neq\mathbb{Q}$, of which part 1) of the main theorem above
is a corollary. 

Next, let us discuss the possibility of extending the methods to other
Shimura varieties. As mentioned above, everything up until section
\ref{sub:3.3 excision} generalizes e.g. to the case of compact PEL
Shimura varieties with a nonvanishing Hasse invariant (or indeed an
affine generalized ordinary locus), however everything after that
depends, in its current form, heavily on the specifics of our moduli problem (in particular its "${\rm{GL}}_{2}$-nature"). We expect that the methods should extend (modulo some issues with the cusps, which we understand have been resolved) to prove the analogous result for overconvergent cusp forms on Hilbert modular varieties.We believe that there should alsobe a different, though more technical, way
of completing the proof using (generalizations of) the results of Shin \cite{Shin} and
a comparison of trace formulas in $p$-adic (rigid) and $\ell$-adic
(etale) cohomology, which should allow for a substantial generalization
of our results. We are currently working out the details for some
unitary Shimura varieties studied by Harris-Taylor \cite{HaTa} and
Taylor-Yoshida \cite{TaYo} where the geometry is well documented. We believe that the techniques of rigid cohomology, in view of its direct relation to overconvergent automorphic forms, will be useful more generally in the theory of $p$-adic automorphic forms. As an example of this, let us mention that a central part of the spectacular recent announcement of Harris, Lan, Taylor and Thorne associating Galois representations to regular algebraic cuspidal automorphic representations of ${\rm GL}_{n}$ is the realization of a certain system of Hecke eigenvalues coming from an Eisenstein series inside a rigid cohomology group analogous to the ones studied in this paper and by Coleman. Finally, we should mention that after this paper had been submitted for publication we were made aware of ongoing work of Tian and Xiao \cite{TiXi} on classicality for overconvergent Hilbert modular forms when $p$ is unramified in $F$. They make a detailed study of the Ekedahl-Oort stratification and obtain very complete results about the structure of the rigid cohomology of the ordinary locus as a Hecke module in order to deduce classicality for small slope overconvergent Hilbert modular forms using techniques similar to those of this paper.

\subsection{Acknowledgments}
The author would like to thank his PhD supervisor Kevin Buzzard for suggesting
this problem and for his constant help and encouragement during every
aspect of this project. He would also like to thank his second supervisor
Toby Gee for valuable advice during the write-up, as well as Wansu
Kim, Christopher Lazda, James Newton, Shu Sasaki and Teruyoshi Yoshida for many helpful
discussions relating to this work, and Francesco Baldassarri and Bernard
Le Stum for answering questions about rigid and overconvergent de
Rham cohomology. The author wishes to thank the Engineering and Physical Sciences Research Council for supporting
him throughout his doctoral studies. It is also a pleasure to
thank the Fields Institute, where part of the write-up
of this paper was done, for their support and hospitality as well
as for the excellent working conditions provided. Finally the author wishes to thank the anonymous referee for correcting some typos and for insightful comments. In the first version of this paper, $p$ was assumed to be inert in $F$. The author wishes to sincerely thank the referee for pointing out that the methods should extend to the case of $p$ unramified, and for urging the author to investigate the general case.

\section{The groups and the Shimura varieties\label{sec:1}}

Throughout this article we fix a rational prime $p$.

\subsection{\label{sub:Groups-and-algebras}Groups and algebras}

Let $F$ be a totally real field of degree $d$ over $\mathbb{Q}$,
with ring of integers $\mathcal{O}_{F}$ in which $p$ splits as
\[
p=\mathfrak{p}_{1}^{e_{1}}...\mathfrak{p}_{r}^{e_{r}}
\]
Write $f_{i}$ for the inertia degree of $\mathfrak{p}_{i}$ and put $d_{i}=e_{i}f_{i}$. We let $B$ denote a totally indefinite quaternion algebra over $F$,
which we in addition assume to be split at all $\mathfrak{p}_{i}$ and a division algebra,
i.e. not equal to $M_{2/F}$. Denote by $\mathcal{O}_{B}$ a maximal
order of $B$, which will be fixed throughout the paper. We will also fix an isomorphism $\mathcal{O}_{B}\otimes_{\mathbb{Z}_{p}}\mathcal{O}_{F_{p}}\cong M_{2}(\mathcal{O}_{F_{p}})$ (where $F_{p}=F\otimes_{\mathbb{Q}}\mathbb{Q}_{p}$), via the transpose this gives an isomorphism $\mathcal{O}_{B^{op}}\otimes_{\mathbb{Z}_{p}}\mathcal{O}_{F_{p}}\cong M_{2}(\mathcal{O}_{F_{p}})$. The group
of invertible elements $\mathcal{O}_{B}^{\times}$ is the $\mathcal{O}_{F}$-points
of an algebraic group, and we denote by $G$ the restriction of scalars
of this group to $\mathbb{Z}$, i.e. for any ring $R$ :
\[
G(R)=(\mathcal{O}_{B}\otimes_{\mathbb{Z}}R)^{\times}
\]
 The reduced norm map ${\rm det}\,:\,\mathcal{O}_{B}^{\times}\rightarrow\mathcal{O}_{F}^{\times}$
defines a homomorphism of algebraic groups ${\rm det}\,:\, G\rightarrow Res_{\mathbb{Z}}^{\mathcal{O}_{F}}\mathbb{G}_{m}$.
We define an algebraic subgroup $G^{\star}\subseteq G$ by the Cartesian
diagram
\[
\xymatrix{G_{/\mathbb{Z}}^{\star}\ar[r]\ar[d]^{{\rm det}} & G_{/\mathbb{Z}}\ar[d]^{{\rm det}}\\
\mathbb{G}_{m/\mathbb{Z}}\ar[r] & Res_{\mathbb{Z}}^{\mathcal{O}_{F}}\mathbb{G}_{m/\mathcal{O}_{F}}
}
\]
where the lower horizontal map is the injection given on $R$-points
by $R^{\times}\rightarrow(\mathcal{O}_{F}\otimes_{\mathbb{Z}}R)^{\times}$,
$r\mapsto1\otimes r$. Note that the $R$-points of $G^{\star}$ are
\[
G^{\star}(R)=\left\{ g\in(\mathcal{O}_{B}\otimes_{\mathbb{Z}}R)\,|\,{\rm det}(g)\in R^{\times}\right\} 
\]
Let $E$ be a finite extension of $F$ that splits $G$.
We fix a Borel subgroup of $G$ over $E$ and by intersecting it with
$G^{\star}$ one gets a Borel $B^{\star}$ of $G^{\star}$. We fix
maximal tori $T$ and $T^{\star}$ of $G$ and $G^{\star}$ defined
over $E$. Since $B$ is split at all $\mathfrak{p}_{i}$ we note that 
\[
G^{\star}(\mathbb{Z}_{p})=\left\{ g\in{\rm GL}_{2}(\mathcal{O}_{F_{p}})\,|\,{\rm det}(g)\in\mathbb{Z}_{p}^{\times}\right\} 
\]
 
\[
G^{\star}(\mathbb{Q}_{p})=\left\{ g\in{\rm GL}_{2}(F_{p})\,|\,{\rm det}(g)\in\mathbb{Q}_{p}^{\times}\right\} 
\]
Let us once and for all fix embeddings of $\overline{\mathbb{Q}}$ into $\mathbb{C}$ and $\overline{\mathbb{Q}}_{p}$. This allows us to identify the archimedean places of $F$ with the embeddings of $F$ into $\overline{\mathbb{Q}}_{p}$. We will enumerate them using pairs $(i,j)$ with $1\leq i\leq r$ and $1\leq j \leq d_{i}$ (here $i$ is of course the same $i$ as in $\mathfrak{p}_{i}$).
 The $\mathbb{C}$-points of $G^{\star}$ and $T^{\star}$ may then be
described as follows
\[
G^{\star}(\mathbb{C})=\left\{ (g_{ij})\in\prod_{i,j}GL_{2}(\mathbb{C})\mid{\rm det}(g_{ij})={\rm det}(g_{i^{\prime}j^{\prime}})\,\forall (i,j)\neq (i^{\prime},j^{\prime})\right\} 
\]
\[
T^{\star}(\mathbb{C})=\left\{ (g_{ij})\in G^{\star}(\mathbb{C})\mid g_{ij}\:{\rm diagonal}\,\forall (i,j)\right\} 
\]
 The center of $\mathcal{O}_{B}$ is $\mathcal{O}_{F}$, hence the
center of $G$ is $Res_{\mathbb{Z}}^{\mathcal{O}_{F}}\mathbb{G}_{m}$
and the center $Z^{\star}$ of $G^{\star}$ is $\mathbb{G}_{m}$.
We have (with the above description of $G^{\star}(\mathbb{C})$)
\[
Z^{\star}(\mathbb{C})=\left\{ (\lambda I)_{ij}\in G^{\star}(\mathbb{C})\mid\lambda\in\mathbb{C}^{\times}\right\} 
\]

The derived group of $\mathcal{O}_{B}^{\times}$ (as an algebraic
group over $\mathcal{O}_{F}$) consists of the elements of reduced
norm 1. It follows that the derived subgroup of both $G$ and $G^{\star}$
is the kernel of the reduced norm map ${\rm det}$. As it is the same
for both $G$ and $G^{\star}$, we will denote it by $G^{der}$. We
have 
\[
G^{der}(\mathbb{C})=\prod_{i,j}{\rm SL}_{2}(\mathbb{C})
\]
 We fix a maximal torus $T^{der}$ of $G^{der}$ over $E$ and make
it so that
\[
T^{der}(\mathbb{C})=T^{\star}(\mathbb{C})\cap G^{der}(\mathbb{C})=\left\{ \left(\begin{array}{cc}
a_{ij}\\
 & a_{ij}^{-1}
\end{array}\right)\in\prod_{i,j}{\rm SL}_{2}(\mathbb{C})\mid a_{ij}\in\mathbb{C}^{\times}\right\} 
\]

\subsection{Representation theory of $G^{\star}$}

In this section we describe the finite dimensional representation
theory of $G^{\star}$ and its weights and central characters. As
with any reductive group, its finite dimensional irreducible representations
are given by a finite dimensional irreducible representation of its
derived group together with a matching central character, where matching
means that the representation and the central character must agree
on the intersection between the derived group and the center.
\begin{rem}
The intersection of $G^{der}(\mathbb{C})$ and $Z^{\star}(\mathbb{C})$
is $\left\{ \pm I\right\} =\left\{ \pm\left(\begin{array}{cc}
1\\
 & 1
\end{array}\right)_{ij}\in\prod_{i,j}{\rm SL}_{2}(\mathbb{C})\right\} $, so we need to check compatibility on the element $-I$.
\end{rem}
The representations of ${\rm SL}_{2}$ are well known and gives us
the following:

\begin{prop}
The irreducible finite dimensional representations of $G^{der}(\mathbb{C})$
are parametrized by $d$-tuples of non-negative integers $(k_{11},...,k_{rd_{r}})$,
corresponding to the representation
\[
\bigotimes_{i,j}Sym^{k_{ij}}(Sd_{ij})
\]
 where $Sd_{ij}$ is the representation given by projection $G^{der}(\mathbb{C})=\prod_{i,j}{\rm SL}_{2}(\mathbb{C})\rightarrow{\rm SL}_{2}(\mathbb{C})$
onto the $(i,j)$-th factor together with the standard (left) representation
of ${\rm SL}_{2}(\mathbb{C})$ on $\mathbb{C}^{2}$. All these representations
can be defined over any field extension of $F$ that splits $B$,
in particular $E$. The element $-I$ acts on \textup{$\bigotimes_{i,j}Sym^{k_{ij}}(Sd_{ij})$
by $(-1)^{\sum{k_{ij}}}$.}
\end{prop}

Since $Z^{\star}\cong\mathbb{G}_{m}$ we deduce

\begin{cor}
The irreducible finite dimensional representations of $G^{\star}(\mathbb{C})$
are parametrized by $(d+1)$-tuples of integers $(k_{11},...,k_{rd_{r}},w)$,
with $k_{ij}\geq0$ for all $i,j$ and $w\equiv\sum k_{ij}$ mod $2$,
and this corresponds to the representation
\[
\left(\bigotimes_{i,j}Sym^{k_{ij}}(Sd_{ij})\right)\otimes{\rm det}^{\left(w-\sum k_{ij}\right)/2}
\]
 Here, similar to before $Sd_{ij}$ is the representation given by
projection $G^{\star}(\mathbb{C})\subseteq\prod_{i,j}{\rm GL}_{2}(\mathbb{C})\rightarrow{\rm GL}_{2}(\mathbb{C})$
onto the $(i,j)$-th factor together with the standard (left) representation
of ${\rm GL}_{2}(\mathbb{C})$ on $\mathbb{C}^{2}$, and ${\rm det}$
is the reduced norm character. \textup{$\bigotimes_{i,j}Sym^{k_{ij}}(Sd_{ij})$
corresponds to $(k_{11},...,k_{rd_{r}},\sum k_{ij})$ and ${\rm det}$ corresponds
to $(0,...,0,2)$. }As before, all representations can be defined
over any extension of $F$ that splits $B$, in particular $E$.
\end{cor}

We have a similar description of the characters of $T^{\star}$ and
$T^{der}$:

\begin{prop}
The characters of $T^{der}(\mathbb{C})$ are parametrized by $d$-tuples
of integers $(k_{11},...,k_{rd_{r}})$ and the characters of $T^{\star}(\mathbb{C})$
are parametrized by $d+1$-tuples $(k_{11},...,k_{rd_{r}},w)$ of integers
such that $w\equiv\sum k_{ij}$ mod $2$. We will denote the corresponding
characters by $\chi(k_{11},...,k_{rd_{r}})$ resp. $\chi(k_{11},...,k_{rd_{r}},w)$.
\end{prop}

Next we wish to describe a representation which will be important
in what follows. This is the representation, defined over $\mathbb{Z}$,
given by the standard left action of $G^{\star}(R)\subseteq(\mathcal{O}_{B}\otimes_{\mathbb{Z}}R)^{\times}$
on $\mathcal{O}_{B}\otimes_{\mathbb{Z}}R$, and we will denote it
$Sd$. Over any extension that the $Sd_{ij}$ are defined over, it
splits non-canonically as $Sd=\bigoplus_{i,j}\left(Sd_{ij}\oplus Sd_{ij}\right)$.
The representations we will be working with are certain summands of the symmetric powers
$Sym^{k}(Sd)$ and certain of its subrepresentations. We have that
\[
Sym^{k}(Sd)=Sym^{k}\left(\bigoplus_{i}\bigoplus_{j}\left(Sd_{ij}\oplus Sd_{ij}\right)\right)=\bigoplus_{(k_{1},..,k_{r})}\bigotimes_{i}Sym^{k_{i}}\left(\bigoplus_{j}\left(Sd_{ij}\oplus Sd_{ij}\right)\right)
\]
where the sum in the furthermost right hand side is taken over all
$r$-tuples of non-negative integers $(k_{1},...,k_{r})$
such that $\sum k_{i}=k$. Put $Sd_{i}=\bigoplus_{j}\left(Sd_{ij}\oplus Sd_{ij}\right)$. The representations $\bigotimes_{i}Sym^{k_{i}}Sd_{i}$ are the representations that we will be interested. Note that they are defined over $\mathbb{Q}_{p}$. We have 
\[
Sym^{k_{i}}(Sd_{i})=\bigoplus_{(k_{i1},...,k_{id_{i}})}\bigotimes_{j}Sym^{k_{ij}}\left(Sd_{ij}\oplus Sd_{ij}\right)
\]
and
\[
Sym^{k_{ij}}(Sd_{ij}\oplus Sd_{ij})=\bigoplus_{0\leq k_{ij}^{\prime} \leq k_{ij}}\left(Sym^{k_{ij}^{\prime}}(Sd_{ij})\otimes Sym^{k_{ij}-k_{ij}^{\prime}}(Sd_{ij})\right)
\]
Moreover, we have 
\[
Sym^{k_{ij}-k_{ij}^{\prime}}(Sd_{ij})\otimes Sym^{k_{ij}^{\prime}}(Sd_{ij})=\bigoplus_{0\leq a_{ij}\leq k_{ij}/2}\left(Sym^{k_{ij}-2a_{ij}}(Sd_{ij})\otimes{\rm det}^{a_{ij}}\right)
\]
where the $a_{ij}$ are integers. Putting it together we have 
\[
Sym^{k_{i}}(Sd_{i})=\bigoplus_{(k_{i1},..,k_{id_{i}},a_{i1},...,a_{id_{i}})}\left(\bigotimes_{j}Sym^{k_{ij}-2a_{ij}}(Sd_{ij})\right)\otimes{\rm det}^{\sum a_{ij}}
\]
 with the $k_{ij}$ and $a_{ij}$ as above.

\subsection{Shimura varieties defined by $G^{\star}$ and their integral models}

In this section we briefly recall some more or less well known constructions, though as far as author is aware of they are not explicitly stated in the literature when $p$ ramifies in $F$. When $p$ is unramified see
e.g. \cite{Kott}, \cite{Mil05} and \cite{Lan}. $B$ carries an
involution $b\mapsto b^{\ast}$ of the first kind. Consider the opposite
$\mathbb{Q}$-algebra $B^{op}$ with involution $b\mapsto b^{\ast}$,
with the natural left action on $B$ . Pick $\xi\in B$ such that
$\xi^{\ast}=-\xi$, and define a $B^{op}$-involution on $B$ by $(x,y)=Tr_{F/\mathbb{Q}}Tr_{B/F}(x^{\ast}\xi y)$,
where $Tr_{B/F}$ is the reduced trace and $Tr_{F/\mathbb{Q}}$ is
the field trace. Together with the homomorphism $h\,:\,\mathbb{C}\rightarrow End_{B^{op}\otimes_{\mathbb{Q}}\mathbb{R}}(B\otimes_{\mathbb{Q}}\mathbb{R})=M_{2}(F\otimes_{\mathbb{Q}}\mathbb{R})$
given by 
\[
a+bi\mapsto\left(\begin{array}{cc}
1\otimes a & -1\otimes b\\
1\otimes b & 1\otimes a
\end{array}\right)
\]
This defines a rational PEL-datum of type C, and hence a Shimura datum
whose group is $G^{\star}$ acting on the disconnected Hermitian symmetric
domain $(\mathcal{H}^{+})^{d}\sqcup(\mathcal{H}^{-})^{d}$, where
$\mathcal{H}^{+}$ is the upper and $\mathcal{H}^{-}$ is the lower
half plane. The associated Shimura varieties are moduli spaces for
abelian varieties with extra structures and are defined over the reflex
field $\mathbb{Q}$. For a given neat compact open subgroup $K\subseteq G^{\star}(\mathbb{A}^{\infty})$
we denote the corresponding Shimura variety by $Sh_{K}$. For primes not dividing the level or the discriminant
of $F$ or $B$, the canonical models of $Sh_{K}$ have good reduction. In the more general case when $p$ is allowed to divide the discriminant of $F$ but not the discriminant of $B$, integral models may be constructed and studied by copying the methods of Deligne and Pappas \cite{DePa} in the Hilbert case. We will now very briefly recall the construction of these integral models and a few of their properties.

Fix an open compact subgroup $K=K^{p}K_{p}\subseteq G^{\star}(\mathbb{A}^{\infty})$
such that $K_{p}=G^{\star}(\mathbb{Z}_{p})$, $K^{p}$ will be specified below and fix a fractional ideal $\mathfrak{c}$ of $F$ (without loss of generality coprime to $p$). We denote the totally positive elements of $\mathfrak{c}$ by $\mathfrak{c}^{+}$. Let $N\geq 5$ be an integer, coprime to $p$. Define a functor $\mathcal{X}^{DP}$ sending a locally Noetherian $\mathbb{Z}_{p}$-scheme $S$ to the set of isomorphism classes of quadruples $(A,\iota,\phi,\eta)$ where
\begin{enumerate}

\item $A/S$ is an abelian scheme of dimension $2d$
\smallskip
\item $\iota\,:\,\mathcal{O}_{B^{op}}\rightarrow End_{S}(A)$
is a ring homomorphism
\smallskip
\item $\phi$ is an $\mathcal{O}_{F}$-linear homomorphism of $\mathfrak{c}$ into the sheaf of symmetric  homomorphisms $\lambda\,:\,A\rightarrow A^{\vee}$ satisfying $i(b)^{\vee}\circ\lambda=\lambda\circ i(b^{\ast})$ (as quasi-isogenies) for all $b\in \mathcal{O}_{B^{op},(p)}$. We require that $\phi$ maps $\mathfrak{c}^{+}$ to polarizations, and that the map $A\otimes\mathfrak{c}\rightarrow A^{\vee}$ induced by $\phi$ is an isomorphism (the "Deligne-Pappas condition").
\smallskip
\item $\eta$ is an $\mathcal{O}_{B^{op}}$-linear closed immersion $\mathcal{O}_{B^{op}}/N\mathcal{O}_{B^{op}}\rightarrow A[N]$ of group schemes.
\end{enumerate}

By standard methods, this functor is represented by a projective scheme over $\mathbb{Z}_{p}$ which we also denote $\mathcal{X}^{DP}$. Properness is the only thing that differs from the Hilbert case. It follows (via the valuative criterion of properness) from the potentially good reduction of pairs $(A,i)$ over the fraction field of a discrete valuation ring (see the Proposition in \S 6 of \cite{Bou}; the proof there does not require the "Rapoport condition" that is also assumed in their definition of an abelian scheme with an $\mathcal{O}_{B^{op}}$-action). The generic fibre of $\mathcal{X}^{DP}$ is the canonical model of $Sh_{K}$ base changed to $\mathbb{Q}_{p}$; we will denote it by $X$. From now on, we will simply write $A$ for an isomorphism class of quadruples as above.

\begin{rem} \label{rem : local geometry}
Assume that $A$ is a quadruple as above. It defines a principally polarized $p$-divisible group $A[p^{\infty}]$ of height $4d$ and dimension $2d$ with an action of $\mathcal{O}_{B^{op},p}=\mathcal{O}_{B^{op}}\otimes_{\mathbb{Z}}\mathbb{Z}_{p}\cong M_{2}(\mathcal{O}_{F,p})$. By Morita equivalence, this is equivalent to a principally polarized
$p$-divisible group $\mathcal{G}_{A}$ of height $2d$ and dimension
$d$ with an action of $\mathcal{O}_{F,p}$. The deformations of $\mathcal{G}_{A}$ controls the local geometry of the special fibre of $\mathcal{X}^{DP}$ by Serre-Tate theory. This
is identical to the situation in the Hilbert case, and we may hence use the local models of \cite{DePa} to study the geometry of $\mathcal{X}^{DP}$. In particular, the fibres of $\mathcal{X}^{DP}$ are normal.

Let us now specify the tame level $K^{p}$ used above. It is analogous to the choice of $"\Gamma_{1}(\mathfrak{c},N)"$-level structure often made in the literature on overconvergent Hilbert modular forms. Let $c\in\mathbb{A}_{F}^{\infty}$
be a fixed representative of a double coset in $F_{+}^{\times}\backslash\mathbb{A}_{F}^{\infty}/\hat{\mathcal{O}}_{F}^{\times}$,
where $F_{+}^{\times}$ denotes the totally positive elements of $F^{\times}$
and $\hat{\mathcal{O}}_{F}^{\times}=(\mathcal{O}_{F}\otimes_{\mathbb{Z}}\hat{\mathbb{Z}})^{\times}$.
This $c$ corresponds to $\mathfrak{c}$ and is relatively prime to $p$. Define 
\[
K_{1}^{G}(N)=\left\{ g\in{\rm GL}_{2}(\mathbb{A}_{F}^{\infty})\mid g\equiv\left(\begin{array}{cc}
1 & \ast\\
0 & \ast
\end{array}\right)\,{\rm mod}\, N\right\} 
\]
 Finally, we put 
\[
K_{1}(c,N)=G^{\star}(\mathbb{A}_{\mathbb{Q}}^{\infty})\cap\left(\begin{array}{cc}
c & 0\\
0 & 1
\end{array}\right)K_{1}^{G}(N)\left(\begin{array}{cc}
c & 0\\
0 & 1
\end{array}\right)^{-1}
\]
 where the intersection takes place in ${\rm GL}_{2}(\mathbb{A}_{F}^{\infty})$.
As $c$ and $N$ are prime to $p$, $K_{1}(c,N)_{p}=G^{\star}(\mathbb{Z}_{p})$.
For $N$ as above $K_{1}(c,N)$ is neat, and we put $K=K_{1}(c,N)$. As ${\rm det}(K)=\hat{\mathbb{Z}}^{\times}$, $X$ is geometrically connected. By the usual trick using Zariski's connectedness principle, the special fiber of $\mathcal{X}^{DP}$ is geometrically connected and hence geometrically irreducible by normality. 
\end{rem}

$\mathcal{X}^{DP}$ is not smooth. We will need the fact that we can resolve the singularities of $\mathcal{X}^{DP}$ after a ramified extension of valuation rings. This is identical to the situation in the Hilbert case as studied in \cite{Sas} (following work of Pappas and Rapoport on local models) so we will be rather brief. The author wishes to thank Shu Sasaki for explaining his work to him. We will follow \cite{Sas} closely in what follows. Let $L\subseteq\overline{\mathbb{Q}}_{p}$ be a finite extension of $\mathbb{Q}_{p}$ that contains the image of every embedding $F_{\mathfrak{p}_{i}}\hookrightarrow \overline{\mathbb{Q}}_{p}$ for every $i$, and let $\mathcal{O}_{L}$ denote its ring of integers. Let $\pi_{L}$ denote a fixed uniformizer of $L$, and let $L^{ur}$ denote the maximal unramified subfield of $L$ (and similarly for other $p$-adic fields).

Before we give the new moduli problem we need some more notation. For each $i$, fix a uniformizer $\pi_{i}$ of $F_{\mathfrak{p}_{i}}$ satisfying an Eisenstein polynomial $E_{i}(u)\in \mathcal{O}_{F_{\mathfrak{p}_{i}}^{ur}}[u]$. Moreover, we put $S_{i}={\rm Hom}_{\mathbb{Z}_{p}}(\mathcal{O}_{F_{\mathfrak{p}_{i}}^{ur}},\mathcal{O}_{L^{ur}})$. For every $\sigma\in S_{i}$ we put $E_{i,\sigma}(u)=\sigma(E_{i}(u))\in \mathcal{O}_{L^{ur}}[u]$ and let $\{\pi_{\sigma}(1),...,\pi_{\sigma}(e_{i})\}$ denote its set of of roots in $L$. Continuing, we denote by $\{\sigma(j)\}_{\sigma\in S_{i},1\leq j\leq e_{i}}$ the $d_{i}$ embeddings of $F_{\mathfrak{p}}$ into $L$, where $\sigma(j)$ is defined $\sigma(j)|_{\mathcal{O}_{F_{\mathfrak{p}_{i}}^{ur}}}=\sigma$ and that it maps $\pi_{i}$ to $\pi_{\sigma}(j)$. We have
\[
\mathcal{O}_{B^{op}}\otimes_{\mathbb{Z}}\mathcal{O}_{L}\cong M_{2}(\mathcal{O}_{F,p})\otimes_{\mathbb{Z}_{p}}\mathcal{O}_{L}\cong M_{2}(\mathcal{O}_{F}\otimes_{\mathbb{Z}}\mathcal{O}_{L})
\]
and
\[
\mathcal{O}_{F}\otimes_{\mathbb{Z}}\mathcal{O}_{L}\cong \bigoplus_{i}\mathcal{O}_{F_{\mathfrak{p}_{i}}}\otimes_{\mathbb{Z}_{p}}\mathcal{O}_{L}\cong \bigoplus_{i}\left(\mathcal{O}_{F_{\mathfrak{p}_{i}}}\otimes_{\mathcal{O}_{F_{\mathfrak{p}_{i}}^{ur}}}\left( \mathcal{O}_{F_{\mathfrak{p}_{i}}^{ur}} \otimes_{\mathbb{Z}_{p}}\mathcal{O}_{L}\right)\right)\cong
\]
\[
\cong \bigoplus_{i}\left(\mathcal{O}_{F_{\mathfrak{p}_{i}}}\otimes_{\mathcal{O}_{F_{\mathfrak{p}_{i}}^{ur}}}\left(\bigoplus_{\sigma\in S_{i}}\mathcal{O}_{L}\right)\right)\cong \bigoplus_{i}\bigoplus_{\sigma\in S_{i}}\left(\mathcal{O}_{F_{\mathfrak{p}_{i}}}\otimes_{\mathcal{O}_{F_{\mathfrak{p}_{i}}^{ur},\sigma}}\mathcal{O}_{L}\right)
\]
Put $\mathcal{O}_{i,\sigma}=\mathcal{O}_{F_{\mathfrak{p}_{i}}}\otimes_{\mathcal{O}_{F_{\mathfrak{p}_{i}}^{ur}}} \mathcal{O}_{L}$, then we have that $\mathcal{O}_{B^{op}}\otimes_{\mathbb{Z}}\mathcal{O}_{L}\cong \bigoplus_{i}\bigoplus_{\sigma\in S_{i}}M_{2}(\mathcal{O}_{i,\sigma})$. Let $A$ be an element of $\mathcal{X}^{DP}(S)$. Then we get decompositions 
\[
H^{dR}_{1}(A/S)=\bigoplus_{i}\bigoplus_{\sigma\in S_{i}}H_{1}^{dR}(A/S)_{i,\sigma}
\]
\[
Lie(A^{\vee}/S)^{\vee}=\bigoplus_{i}\bigoplus_{\sigma\in S_{i}}Lie(A^{\vee}/S)^{\vee}_{i,\sigma}
\]
where $H_{1}^{dR}(A/S)_{i,\sigma}$ is an $M_{2}(\mathcal{O}_{S}\otimes\mathcal{O}_{i,\sigma})$-module which is locally free of rank 4 as an $\mathcal{O}_{S}\otimes\mathcal{O}_{i,\sigma}$-module, and $Lie(A^{\vee}/S)^{\vee}_{i,\sigma}$ is an $M_{2}(\mathcal{O}_{S}\otimes\mathcal{O}_{i,\sigma})$-module that is, Zariski locally on $S$, a locally free direct summand of  $H_{1}^{dR}(A/S)_{i,\sigma}$ of rank $2e_{i}$ as an $\mathcal{O}_{S}$-module.

We define a functor $\mathcal{X}^{PR}$ from the category of locally Noetherian schemes over $\mathcal{O}_{L}$ to sets by letting, for $S$ a scheme over $\mathcal{O}_{L}$, $\mathcal{X}^{PR}(S)$ be the set of isomorphism classes of data 
\[
(A,(F_{i,\sigma}(j))_{i,\sigma,j})
\]
where
\begin{enumerate}
\item $A\in \mathcal{X}^{DP}(S)$
\medskip
\item For every $i$ and $\sigma\in S_{i}$, we have a filtration
\[
0=F_{i,\sigma}(0)\subseteq F_{i,\sigma}(1)\subseteq ... \subseteq F_{i,\sigma}(e_{i})=Lie(A^{\vee}/S)^{\vee}_{i,\sigma}
\]
of $M_{2}(\mathcal{O}_{S}\otimes\mathcal{O}_{i,\sigma})$-modules such that
\smallskip
\begin{enumerate}

\item each $F_{i,\sigma}(j)$ is, Zariski locally on $S$, a direct summand of $Lie(A^{\vee}/S)^{\vee}_{i,\sigma}$ of rank $2j$ as an $\mathcal{O}_{S}$-module and

\item on the quotient $F_{i,\sigma}(j)/F_{i,\sigma}(j-1)$ ($j\geq 1$), which is a locally free $\mathcal{O}_{S}$-module of rank $2$, $\mathcal{O}_{B^{op}}$ acts via 
\[
\mathcal{O}_{B^{op}}\hookrightarrow M_{2}(\mathcal{O}_{F_{\mathfrak{p}_{i}}})\overset{\sigma(j)}{\hookrightarrow}M_{2}(\mathcal{O}_{L})\rightarrow M_{2}(\mathcal{O}_{S}) 
\]
\end{enumerate}

\end{enumerate}
Using Morita equivalence the proofs of \cite{Sas} carry over verbatim and shows that the forgetful natural transformation $\mathcal{X}^{PR}\rightarrow \mathcal{X}^{DP}_{\mathcal{O}_{L}}$ (subscript denoting base change) is relatively representable by a projective morphism and hence that $\mathcal{X}^{PR}$ is representable. We will denote the representing object by $\mathcal{X}^{PR}$ as well. As in \cite{Sas}, $\mathcal{X}^{PR}$ is smooth over $\mathcal{O}_{L}$ (this is proved using Grothendieck-Messing theory). Moreover, the morphism $\mathcal{X}^{PR}\rightarrow \mathcal{X}^{DP}_{\mathcal{O}_{L}}$ is an isomorphism over the Rapoport locus (which coincides with the smooth locus of $ \mathcal{X}^{DP}_{\mathcal{O}_{L}}$), which includes the ordinary locus in the special fibre and the whole generic fibre. In particular, the generic fibre of $\mathcal{X}^{PR}$ is $X_{L}$ and the fibres of $\mathcal{X}^{PR}$ are geometrically connected.

Next we will add level structure at $p$. Define two subgroups $K_{0}(p)$,
$K_{0}^{0}(p)$ of $G^{\star}(\mathbb{Z}_{p})$ by 
\[
K_{0}(p)=\left\{ g\in G^{\star}(\mathbb{Z}_{p})\mid g\equiv\left(\begin{array}{cc}
\ast & \ast\\
0 & \ast
\end{array}\right)\,{\rm mod}\, p\right\} 
\]
\[
K_{0}^{0}(p)=\left\{ g\in G^{\star}(\mathbb{Z}_{p})\mid g\equiv\left(\begin{array}{cc}
\ast & 0\\
0 & \ast
\end{array}\right)\,{\rm mod}\, p\right\} 
\]
 We let $Y$ resp. $Z$ be the base change of the canonical model of $Sh_{K^{p}K_{0}(p)}$ resp. $Sh_{K^{p}K_{0}^{0}(p)}$ to $\mathbb{Q}_{p}$. $Y$
resp. $Z$ parametrize pairs $(A,H)$ resp. triples $(A,H_{1},H_{2})$,
where $A$ is a point of $X$ and $H$, $H_{1}$ and $H_{2}$ are
finite flat (in fact etale) $\mathcal{O}_{B^{op}}$-stable subgroups of $A$ of
rank $p^{d}$ which are killed by $p$ and isotropic with respect
to the polarization. Moreover we require that $H_{1}\cap H_{2}=0$. The
relative representability of these moduli problems over $X$ may be
shown by standard methods (they are closed subschemes of Grassmannians). We have finite etale morphisms $Z\rightrightarrows Y\rightarrow X$
forgetting $H_{1}$ resp. $H_{2}$ resp. $H$. Since ${\rm det}(K^{p}K_{0}(p))={\rm det}(K^{p}K_{0}^{0}(p))=\hat{\mathbb{Z}}^{\times}$,
$Y$ and $Z$ are geometrically connected.

Let $k_{L}$ denote the residue field of $L$. We will denote the special fibres of $\mathcal{X}^{DP}$ over $\mathbb{F}_{p}$, $k_{L}$ resp. $\overline{\mathbb{F}}_{p}$ by $X_{\mathbb{F}_{p}}^{DP}$ , $X_{k_{L}}^{DP}$ resp. $X_{\overline{\mathbb{F}}_{p}}^{DP}$, and similarly for $\mathcal{X}^{PR}$.

\begin{rem}
We will use the notation $(A,...)$ as above to denote points of the
special and/or generic fibres of our moduli spaces; however we will
also use the notation $\mathcal{A}^{DP}$, $\mathcal{A}^{PR}$, $A$, $A_{\mathbb{F}_{p}}^{DP}$
etc. (analogous to $\mathcal{X}^{DP}$, $\mathcal{X}^{PR}$, $X$, $X_{\mathbb{F}_{p}}^{DP}$ etc.)
to denote the (abelian scheme associated with) the universal object
over the appropriate moduli space. We hope there will be no confusion
arising from this. Occasionally we will use the superscript $^{univ}$ to distinguish the universal object. We will denote the map from the universal object to the moduli space by $\pi$ and the zero section of the universal object by $e$; if there is need to identify which moduli space we are dealing we will use appropriate subscripts; we hope that no confusion will arise from this either.
\end{rem}

\section{Automorphic forms and Hecke operators\label{sec:2}}

\subsection{\label{sub:2.1 Aut vec bundle}Automorphic vector bundles and automorphic
forms}

One way to define holomorphic automorphic forms is to use the automorphic
vector bundle construction, as described e.g. in \cite{Mil}. The
theorem is the following, and only applies in characteristic $0$ and
therefore applies equally well to $X$, $Y$ or $Z$ or any other
neat level. By abuse of notation, we also let $\chi(k_{11},...,k_{rd_{r}},w)$
denote the representation of $B^{\star}$ obtained from $\chi(k_{11},...,k_{rd_{r}},w)$
by letting the unipotent part of $B^{\star}$ act trivially.

\begin{thm}
To any finite dimensional representation of $B^{\star}$ we may functorially
associate a vector bundle on $X$ such that equivariant maps between
representations go to Hecke-equivariant $\mathcal{O}_{X}$-linear
maps. To any finite dimensional representation of $G^{\star}$ we may
functorially associate a vector bundle with an integrable connection.
These bundles and maps are defined over the same fields as the representations
and maps are (they are defined on the base change of canonical model to said field; we base change them to appropriate extensions of $\mathbb{Q}_{p}$), and the construction respects direct sums and tensor
operations, and the rank of the bundle is the dimension of the representation.
We will denote by $W(k_{11},...,k_{rd_{r}},w)$ the line bundle associated
to $\chi(k_{11},...,k_{rd_{r}},w)$ and by $V(k_{11},...,k_{rd_{r}},w)$ the vector
bundle with connection associated to $\left(\bigotimes_{i,j}Sym^{k_{ij}}(St_{ij})\right)\otimes{\rm det}^{\left(w-\sum k_{ij}\right)/2}$.
The representation ${\rm det}$ goes to the Tate twist $\mathbb{Q}_{p}(1)$.\end{thm}

\begin{defn}
An automorphic form of weight $(k_{11},...,k_{rd_{r}},w)$
and level $K_{1}(c,N)$ is a global section of $W(k_{11},...,k_{rd_{r}},w)$
on $X$ (and similarly, changing the level, for $Y$ and $Z$).
\end{defn}

The PEL datum is set up such that the standard representation $Sd$
corresponds to $H_{1}^{dR}(A/X)$, hence $H_{dR}^{1}(A/X)$ corresponds
to $Sd^{\vee}$. $Sd^{\vee}$, as a $T^{\star}$-representation, is
\[
Sd^{\vee}=\left(\chi(1,0,..,0,-1)^{\oplus2}\oplus...\oplus\chi(0,...,0,1,-1)^{\oplus2}\right)\oplus
\]
\[
\oplus\left(\chi(-1,0,..,0,-1)^{\oplus2}\oplus...\oplus\chi(0,...,0,-1,-1)^{\oplus2}\right)
\]
 Another bundle that will occur later is $\Omega_{X}^{d}$. To start
with, $\Omega_{X}^{1}$ corresponds to the dual of the adjoint representation
of $B^{\star}$ on $Lie(G^{\star})/Lie(B^{\star})=\chi(2,0,..,0,0)\oplus...\oplus\chi(0,...,0,2,0)$
(note the trivial central character). Therefore $\Omega_{X}^{d}=\wedge^{d}\Omega_{X}^{1}$
corresponds to $\chi(2,...,2,0)$.

\begin{rem}
1) Let us briefly explain the relation between this and the perhaps
more standard way of defining automorphic forms on $X$, as in e.g.
\cite{Kas1}, from which part of this discussion is taken. This will
also provide an integral structure to our sheaves of automorphic forms (at least after base change to $L$).
Recall our identification of $\mathcal{O}_{B^{op}}\otimes_{\mathbb{Z}}\mathbb{Z}_{p}$
with $M_{2}(\mathcal{O}_{F_{p}})$, and consider the two standard
orthogonal idempotents $e_{1}$ and $e_{2}$ in $M_{2}(\mathcal{O}_{F_{p}})$.
The sheaf $\pi_{\ast}\Omega_{\mathcal{A}^{PR}/\mathcal{X}^{PR}}^{1}=e^{\ast}\Omega_{\mathcal{A}^{PR}/\mathcal{X}^{PR}}^{1}$
injects into $H_{dR}^{1}(\mathcal{A}^{PR}/\mathcal{X}^{PR})$ and corresponds to 
\[
\chi(1,0,..,0,-1)^{\oplus2}\oplus...\oplus\chi(0,...,0,1,-1)^{\oplus2}
\]
on the generic fibre. $\pi_{\ast}\Omega_{\mathcal{A}^{PR}/\mathcal{X}^{PR}}^{1}$ inherits an action of $\mathcal{O}_{B^{op}}$
and carries a scalar action of $\mathbb{Z}_{p}$, hence has an action
$\mathcal{O}_{B^{op}}\otimes_{\mathbb{Z}}\mathbb{Z}_{p}=M_{2}(\mathcal{O}_{F_{p}})$.
Taking the image of $e_{2}$ say (to be consistent with \cite{Kas1}),
we obtain a sheaf $\omega=\omega_{\mathcal{A}^{PR}/\mathcal{X}^{PR}}$ which corresponds to 
\[
\chi(1,0,..,0,-1)\oplus...\oplus\chi(0,...,0,1,-1)
\]
on the generic fibre and still carries an action of $\mathcal{O}_{F_{p}}$. Decomposing
$\omega$ with respect to action of $\mathcal{O}_{F_{p}}$ as in the
Hilbert case, we obtain line bundles $\omega_{ij}$ corresponding to
$\chi(0,...,0,1,0,...,0,-1)$ on the generic fibre (the $1$ in the $(i,j)$-th place), and
automorphic forms of weight $(k_{11},...,k_{rd_{r}})$ are defined as global
sections of $\bigotimes\omega_{ij}^{k_{ij}}$. Note that these correspond
to our automorphic forms of weight $(k_{11},...,k_{rd_{r}},-\sum k_{ij})$,
or rather gives an integral structure to this space. We will see when
we consider Hecke operators that, the way we are used to thinking
about them, automorphic forms of weight $(k_{11},...,k_{rd_{r}})$ with
their usual Hecke action corresponds to global sections of $\left(\bigotimes\omega_{ij}^{k_{ij}-2}\right)\otimes\Omega_{X}^{d}$
(cf. \cite{ChFa} p. 258 for a similar remark in the Siegel case).

2) The central character is only important when we are considering Hecke
operators; the bundles $W(k_{11},...,k_{rd_{r}},w)$ are isomorphic for
fixed $(k_{11},...,k_{rd_{r}})$ but varying $w$. Changing $w$ has the
effect of scaling Hecke operators, which we will see and use explicitly
later. Consequently, we will occasionally just refer to $(k_{11},...,k_{rd_{r}})$
as the weight and sometimes talk about ``an automorphic form of weight
$(k_{11},...,k_{rd_{r}})$'', not specifying $w$, which we will refer
to as ``the central character''. Sometimes we will include $w$
in the weight. We hope that this will not be confusing.

3) As the $W(k_{11},...,k_{rd_{r}},w)$ are isomorphic for fixed $(k_{11},...,k_{rd_{r}})$
and varying $w$ by a canonical isomorphism (see Proposition \ref{pro: scaling hecke})
we may use this isomorphism to define an integral structure on $W(k_{11},...,k_{rd_{r}},w)$
by transport of structure from $W(k_{11},...,k_{rd_{r}},-\sum k_{ij})$.
\end{rem}

\subsection{Ordinary locus, canonical subgroups and overconvergent automorphic
forms}

The Hasse invariant is defined as a section of $\left(\wedge^{2d}e^{\ast}\Omega_{A^{DP}_{\mathbb{F}_{p}}/X^{DP}_{\mathbb{F}_p}}^{1}\right)^{\otimes(p-1)}$
(and can be defined more generally in this fashion for abelian
schemes over arbitrary characteristic $p$ bases) on $X_{\mathbb{F}_{p}}^{DP}$. The ordinary
locus $X_{\mathbb{F}_{p}}^{ord,DP}$ is the locus where the Hasse invariant
does not vanish; its vanishing locus will be denoted $X_{\mathbb{F}_{p}}^{ss,DP}$
(though it is not the supersingular locus except in some low dimensional
cases, we hope this will not cause any confusion). $X_{\mathbb{F}_{p}}^{ord,DP}$
is dense in $X_{\mathbb{F}_{p}}^{DP}$ (as $X_{\mathbb{F}_{p}}^{DP}$ is irreducible and $X_{\mathbb{F}_{p}}^{ord,DP}$ is open). Moreover $X_{\mathbb{F}_{p}}^{ord,DP}$ is smooth (see Remark \ref{rem : local geometry}). The Hodge bundle is ample (see e.g. \cite{LS} proof of Prop. 7.8) and hence $X_{\mathbb{F}_{p}}^{ord,DP}$
is affine (it is the complement of the vanishing locus of a nonzero
section of an ample line bundle on a projective variety). We may make the same definitions for $X_{k_{L}}^{PR}$, giving us $X^{ord,PR}_{k_{L}}$, an open dense affine subset of $X^{PR}_{k_{L}}$, with complement $X^{ss,PR}_{k_{L}}$. Since the map $X^{PR}_{k_{L}}\rightarrow X^{DP}_{k_{L}}$ is an isomorphism on the Rapoport locus, $X^{ord,PR}_{k_{L}}$ is isomorphic to $X^{ord,DP}_{k_{L}}$ and from now on we will drop the superscripts $^{PR}$ or $^{DP}$ from the ordinary locus.

Ultimately we will be interested in rigid-analytic phenomena. When
we have a scheme $S/\mathbb{Q}_{p}$ (or over any extension of complete valued
fields) we will let $S_{an}$ denote its Tate analytification, and
whenever we have an scheme $\mathcal{S}/\mathbb{Z}_{p}$ (or over
any extension of complete valuation rings) we will let $S_{rig}$ denote the
Raynaud generic fibre of the formal completion of $\mathcal{S}$ along its
special fibre. $S_{rig}$ carries a specialization map $sp\,:\, S_{rig}\rightarrow S_{\mathbb{F}_{p}}$.
When $S$ is the generic fiber of $\mathcal{S}$ there is always an
open immersion $S_{rig}\rightarrow S_{an}$ which is an isomorphism
when $\mathcal{S}$ is proper. These notions apply to $X$, $Y$ and
$Z$ and their integral models when they exist. Inside $X_{an}=X_{rig}$,
with respect to $sp\,:\, X_{rig}\rightarrow X_{\mathbb{F}_{p}}$,
we define $X_{rig}^{ord}=sp^{-1}(X_{\mathbb{F}_{p}}^{ord})$ and $X_{rig}^{ss}=sp^{-1}(X_{\mathbb{F}_{p}}^{ss,DP})$,
the ordinary locus resp. non-ordinary locus in $X_{rig}$. Note that we could also have defined them using the Pappas-Rapoport model, but the result would be the same (after base change to $L$).

Let us briefly recall some well known facts about canonical subgroups. By Proposition 3.4 of \cite{AnGa} applied to the formal completion along the special fibers of $\mathcal{A}^{DP}\rightarrow \mathcal{X}^{DP}$  we obtain a partial section $X^{ord}_{rig}\rightarrow Y^{rig}$, $A\mapsto (A,C_{A})$ of the natural map $Y_{rig}\rightarrow X_{rig}$. $C_{A}$ is called the canonical subgroup of $A$. The image of this morphism will be denote by $Y^{ord}_{rig}$. By Theorem 3.5 of \cite{AnGa}, the canonical subgroup overconverges to give a partial section $V\rightarrow Y_{rig}$ of $Y_{rig}\rightarrow X_{rig}$, where $V$ is some strict neighbourhood of $X^{ord}_{rig}$ in $X_{rig}$.

\begin{rem}
$Y_{rig}^{ord}$ is not the full ordinary locus in $Y$; it
is the so-called ordinary-multiplicative locus. There are several ordinary loci in $Y^{rig}$. Somewhat ad
hoc, we will define
\[
Y_{ord}=\{(A,H)\in Y_{rig}\mid A\in X^{ord}_{rig}\,,\,H\cap C_{A}=0\}
\]
$Y_{ord}$ will only be used in an auxiliary role in the construction of the $U_{p}$-operator.
\end{rem}

Next we will define $p$-adic and overconvergent automorphic forms.
We will abuse notation and use $W(k_{11},...,k_{rd_{r}},w)$ etc. to denote
the analytification of those sheaves on $X_{rig}$ etc.

\begin{defn}
An element of $H^{0}(X_{rig}^{ord},W(k_{11},...,k_{rd_{r}},w)$) is called a $p$-adic automorphic form of weight $(k_{11},...,k_{rd_{r}},w)$. An overconvergent automorphic form of weight $(k_{11},...,k_{rd_{r}},w)$ is an element of
\[
H^{0,\dagger}(X_{rig}^{ord},W(k_{11},...,k_{rd_{r}},w))=\lim_{\rightarrow}H^{0}(V,W(k_{11},...,k_{rd_{r}},w))
\]
where the direct limit is taken over any cofinal set of strict neighbourhoods
of $X_{rig}^{ord}$ in $X_{rig}$. Note that by restriction we have
an inclusion
\[
H^{0,\dagger}(X_{rig}^{ord},W(k_{11},...,k_{rd_{r}},w))\subseteq H^{0}(X_{rig}^{ord},W(k_{11},...,k_{rd_{r}},w))
\]
\end{defn}

\subsection{\label{sub:U_p}Hecke operators and $U_{p}$}

We define Hecke operators for our Shimura varieties as in \cite{Kott}
section 6. For us a special role is played by the Hecke operator $U_{p}$,
defined adelically on $Y$ by the double coset 
\[
K^{p}K_{0}(p)\left(\begin{array}{cc}
p\\
 & 1
\end{array}\right)K^{p}K_{0}(p)
\]
 or moduli theoretically by the correspondence
\[
(p_{1},p_{2})\,:\, Z\rightarrow Y\times Y
\]
where $p_{1}$ and $p_{2}$ are the two maps given by
\[
p_{1}(A,H_{1},H_{2})=(A/H_{2},A[p]/H_{2})
\]
\[
p_{2}(A,H_{1},H_{2})=(A,H_{1})
\]

One also has the diamond operators $\left\langle d\right\rangle \,:\, X\rightarrow X$
for $d\in\mathbb{Z}$ with $d$ suitably coprime to $K$ (we will
only need the case $d=p$) defined by $\left\langle d\right\rangle (A)=A/A[d]$.
Note that $A$ and $A/A[d]$ are isomorphic as abelian varieties.

From now on, in this section only, we will only work in the rigid
analytic setting and therefore drop the ``$rig$'' from the notation
in order to ease it. We wish to define operators on $p$-adic and
overconvergent automorphic forms and so want to know that the $U_{p}$-correspondence
restricts to $Y^{ord}$. Let $Z^{ord}=p_{2}^{-1}(Y^{ord})$.

\begin{lem}
$p_{1}(Z^{ord})\subseteq Y^{ord}$\end{lem}

\begin{proof}
Let $(A,H_{1},H_{2})\in Z^{ord}$. By definition $(A,H_{1})\in Y^{ord}$,
so  $H_{1}=C_{A}$. Therefore $A[p]/H_{2}=C_{A/H_{2}}$, hence $p_{1}(A,H_{1},H_{2})=(A/H_{2},A[p]/H_{2})\in Y^{ord}$.
\end{proof}

We may therefore restrict to get a correspondence 
\[
(p_{1},p_{2})\,:\, Z^{ord}\rightarrow Y^{ord}\times Y^{ord}
\]
 Using the isomorphism $X^{ord}\cong Y^{ord}$ we may view this as
a correspondence on $X^{ord}$, and we may simplify $Z^{ord}$ by
noting that the forgetful map $Z\rightarrow Y$ given by $(A,H_{1},H_{2})\mapsto(A,H_{2})$
identifies $Z^{ord}$ with $Y_{ord}=\left\{ (A,H)\mid A\in X^{ord}\,,\, H\cap C_{A}=0\right\} $
in $Y$, so we get a $U_{p}$-correspondence 
\[
(p_{1},p_{2})\,:\, Y_{ord}\rightarrow X^{ord}\times X^{ord}
\]
 with 
\[
p_{1}(A,H)=A/H
\]
\[
p_{2}(A,H)=A
\]

Next we wish to define another $U_{p}$-correspondence, call it $\tilde{U}_{p}$,
which will turn out to be isomorphic to $U_{p}$. We have a map $Fr\,:\, X^{ord}\rightarrow X^{ord}$
given by $Fr(A)=A/C_{A}$. We denote it $Fr$ because it is a lift
of the relative Frobenius in the sense that 
\[
\xymatrix{X^{ord}\ar[r]^{Fr}\ar[d]^{sp} & X^{ord}\ar[d]^{sp}\\
X_{\mathbb{F}_{p}}^{ord}\ar[r]^{Fr} & X_{\mathbb{F}_{p}}^{ord}
}
\]
 commutes. This will be important when we consider rigid cohomology
later. We define $\tilde{U}_{p}$ as the correspondence
\[
(q_{1},q_{2})\,:\, X^{ord}\rightarrow X^{ord}\times X^{ord}
\]
 where 
\[
q_{1}=id
\]
\[
q_{2}=\left\langle p\right\rangle ^{-1}Fr
\]
 
\begin{lem}
Define two morphisms $\alpha\,:\, X^{ord}\rightarrow Y_{ord}$ and
$\beta\,:\, Y_{ord}\rightarrow X^{ord}$ by
\[
\alpha(A)=(A/C_{A},A[p]/C_{A})
\]
\[
\beta(A,H)=A/H
\]
Furthermore, define an automorphism $\left\langle p\right\rangle _{Y}\,:\, Y_{ord}\rightarrow Y_{ord}$
by 
\[
\left\langle p\right\rangle _{Y}(A,H)=\left(\left\langle p\right\rangle (A),\frac{\left\{ a\in A\mid pa\in H\right\} }{A[p]}\right)
\]
Then $\beta\alpha=\left\langle p\right\rangle $ and $\alpha\beta=\left\langle p\right\rangle _{Y}$,
so $\beta$ defines an isomorphism $Y_{ord}\cong X^{ord}$. \end{lem}

\begin{proof}
We have (equalities as points in the moduli spaces)
\[
\beta\alpha(A)=\beta(A/C_{A},A[p]/C_{A})=\frac{A/C_{A}}{A[p]/C_{A}}=\frac{A}{A[p]}=\left\langle p\right\rangle A
\]
 and 
\[
\alpha\beta(A,H)=\alpha(A/H)=\left(\frac{A/H}{A[p]/H},\frac{\left\{ a\in A\mid pa\in H\right\} /H}{A[p]/H}\right)=\left\langle p\right\rangle _{Y}(A,H)
\]
 where the last equality comes from noting that $A[p]/H$ is the canonical
subgroup in $A/H$ and that the map $A/A[p]\rightarrow A$
induced by the $p$-power map on $A$ sends $\frac{\left\{ a\in A\mid pa\in H\right\} }{A[p]}$
to $H$.
\end{proof}

Finally we may prove

\begin{prop}
$U_{p}\cong\tilde{U}_{p}$\end{prop}

\begin{proof}
By the lemma we know that $X^{ord}\cong Y_{ord}$ via $\beta$, so
it suffices to prove that $q_{1}\beta=p_{1}$ and $q_{2}\beta=p_{2}$.
Now 
\[
q_{1}\beta(A,H)=q_{1}(A/H)=A/H=p_{1}(A,H)
\]
 and 
\[
q_{2}\beta(A,H)=q_{2}(A/H)=\left\langle p\right\rangle ^{-1}\left(\frac{A/H}{A[p]/H}\right)=\left\langle p\right\rangle ^{-1}\left(\frac{A}{A[p]}\right)=A=p_{2}(A,H)
\]
\end{proof}

We may therefore denote both correspondences by $U_{p}$. The description
in terms of $Fr$ will prove useful in order to study the slopes of
$U_{p}$.

It remains to extend $U_{p}$ to (small) strict neighbourhoods of
$X^{ord}$. This can be done both from the more classical point of
view, see \cite{Pill} Prop. 4.8.5, or by the overconvergence of the canonical
subgroup. In fact it is well known that the $U_{p}$-correspondence contracts strict neighbourhoods of the ordinary locus. This may be deduced for example by following \cite{Pill} \S 1.2, defining the degree function on $Y_{rig}$ by pullback from $X(2d)$ (in the notation of \cite{Pill} \S 1.2; although the setup there is for the Siegel modular variety for principally polarized abelian varieties, the arguments go through without change for Siegel modular varieties with polarization type of degree prime to  $p$). Then, an argument as in the proof of Proposition 2.3.6 of \cite{Pill} proves the desired contraction property. The correspondences hence induce compact operators
on spaces of overconvergent automorphic forms.

\begin{rem}
\label{rem:factors of p in U_p }1) The Hecke correspondences away
from $p$ preserve the ordinary locus. Hence, again using Prop. 4.8.5
of \cite{Pill}, these correspondences overconverge and define operators
on overconvergent automorphic forms.

2) To properly let a correspondence $s=(s_{1},s_{2})$ act on automorphic
forms of weight $(k_{1},...,k_{d},w)$ one needs also to specify an
isomorphism $s_{1}^{\ast}W(k_{1},...,k_{d},w)\cong s_{2}^{\ast}W(k_{1},...,k_{d},w)$.
This is done in general by the theory of automorphic vector bundles.
To study $p$-divisibility of $U_{p}$, it is preferable though to
have some moduli-theoretic interpretation. It suffices to give such
an isomorphism for $\pi_{univ,\ast}\Omega_{A^{univ}/X}^{1}$ respecting
the action of $M_{2}(\mathcal{O}_{F_{p}})$, as all sheaves of automorphic
forms are constructed from this data, and so we may describe automorphic
forms as ``functions'' a la Katz defined on ``points'' $(A,\omega)$
with $\omega\in H^{0}(A,\Omega_{A}^{1})$. Thus, in order to describe
the action of $U_{p}$ on automorphic forms we need to, given $(A,\omega)$
and $B=A/H\in U_{p}(A)$, functorially associate some $\omega^{\prime}\in H^{0}(A/H,\Omega_{A/H}^{1})$.
This is done by inverting the pullback of differentials along the
isogeny $A\rightarrow A/H$. 

For our second description of $U_{p}$ we may first of all ignore
$\left\langle p\right\rangle ^{-1}$, as it only changes the level
structure away from $p$. The natural map involved is then (a priori)
the isogeny $B\rightarrow B/C_{B}=A$ and it would seem natural to
use pullback of differentials along this isogeny. These definitions
do not agree however, as the composition $B\rightarrow B/C_{B}=A\rightarrow A/H=B$
is multiplication by $p$ which induces multiplication by $p$ on
differentials, so the two definitions disagree by a factor of $p$.
As is standard, we choose the first definition, and modify the second
by the appropriate factor of $p$. This corresponds geometrically
to, rather than using $B\rightarrow A$, using its ``dual'' $A\rightarrow B$
(defined such that the composition both ways are multiplication by
$p$, and related to the dual isogeny via our polarizations). More
explicitly, one has $\tilde{U}_{p}=p^{-\sum k_{ij}}U_{p}$ on $H^{0,\dagger}(X_{rig}^{ord},W(k_{11},...,k_{rd_{r}},-\sum k_{ij}))$
at first, and then scale so that $\tilde{U}_{p}=U_{p}$. Note that
whereas the theory of automorphic vector bundles gives definitions
of Hecke operators for all weights $(k_{11},...,k_{rd_{r}},w)$, we make
this moduli-theoretic definition a priori only for weights of the
form $(k_{11},...,k_{rd_{r}},-\sum k_{ij})$. For general central characters
we scale appropriately to match the theory of automorphic vector bundles,
cf. Proposition \ref{pro: scaling hecke}.
\end{rem}

For the rest of the article we will let $\mathcal{H}_{K}$ denote
the full Hecke algebra of $G^{\star}(\mathbb{A}^{\infty})$ with respect
to the level $K$, and let $\mathcal{H}_{K}^{p}$ denote the full
Hecke algebra of $G^{\star}(\mathbb{A}^{p,\infty})$ with respect to $K^{p}$. Later on when
we consider eigenforms we will fix a commutative subalgebra $\mathcal{H}^{p}\subseteq\mathcal{H}_{K}^{p}$
(which is assumed to be full for primes $\ell\neq p$ for which $B$
is split and $K^{p}$ is maximal) and work with the (commutative)
subalgebra $\mathcal{H}=\mathcal{H}^{p}[U_{p},\left\langle p\right\rangle ]\subseteq\mathcal{H}_{K}$.

For future use we will define two other correspondences at $p$. The
first is the Frobenius correspondence (or really morphism) 
\[
Fr\,:\, X^{ord}\rightarrow X^{ord}\times X^{ord}
\]
with $Fr_{1}=Fr$ and $Fr_{2}=id$. The second is $T_{p}$: 
\[
T_{p}\,:\, Y\rightarrow X\times X
\]
defined by $(T_{p})_{1}(A,H)=A/H$, $(T_{p})_{2}(A,H)=A$. The analytification
of $T_{p}$ preserves the ordinary locus (as ordinariness is preserved
by isogenies) and hence we may restrict, obtaining a correspondence on $X^{ord}$.
As above both of these correspondences overconverge. Given $A\in X^{ord}$
and $\omega\in H^{0}(A,\Omega_{A}^{1})$, we have $Fr(A)=A/C_{A}$
and define a differential $\omega^{\prime}\in H^{0}(A/C_{A},\Omega_{A/C_{A}}^{1})$
by inverse pullback along $A\rightarrow A/C_{A}$. This makes $Fr$
act on automorphic forms by Remark \ref{rem:factors of p in U_p }.
For $T_{p}$ the same discussion as for $U_{p}$ in Remark \ref{rem:factors of p in U_p }
applies to give the action on automorphic forms. We remark that, as
correspondences, $T_{p}=U_{p}+Fr$ (see \cite{Lau1} section 1.6 for the
definition of addition of correspondences) and with the conventions
above $T_{p}$ and $U_{p}+Fr$ also induce the same actions on automorphic
vector bundles.

\subsection{BGG complexes for $G^{\star}$\label{sub:2.4 BGG}}

We wish to compute the BGG complex of the representation $Sym^{k-2d}(Sd)$,
for $k\geq2d$. For BGG complexes see \cite{BGG} for the original
paper and \cite{Hum} for a recent detailed study. For our purpose,
the theorem specialized to our situation is the following (the passage
from semisimple to reductive Lie algebras merely consists of adding
a central character): 

\begin{thm}
\label{thm:reductive BGG}(BGG resolution) If $V$ is the irreducible
representation of the reductive Lie algebra $\mathfrak{g}^{\star}=Lie(G^{\star}(\mathbb{C}))$
of dominant weight $\lambda=(k_{1},...,k_{d},w)$, then we have a
resolution 
\[
0\to C_{d}^{V}\to...\to C_{0}^{V}\to V\to0
\]

with $C_{r}^{V}=\bigoplus_{w\in W^{(r)}}U(\mathfrak{g}^{\star})\otimes_{U(\mathfrak{b}^{\star})}\chi(w(\lambda+\rho)-\rho)$.
The chain complex $C_{\bullet}^{V}$ is a quasi-isomorphic direct
summand of the bar resolution $D_{\bullet}^{V}$ defined by $D_{r}^{V}=U(\mathfrak{g}^{\star})\otimes_{U(\mathfrak{b}^{\star})}(\wedge^{r}(\mathfrak{g}^{\star}/\mathfrak{b}^{\star})\otimes_{\mathbb{C}}V$),
with $\mathfrak{b}^{\star}=Lie(B^{\star}(\mathbb{C}))$.
\end{thm}

Here $W^{(r)}$ denotes the elements in the Weyl group of length $r$.
The Weyl group of $G^{\star}(\mathbb{C})$ is the same as that for
its derived group, hence isomorphic to $\left\{ \pm1\right\} ^{d}$,
and an element $(\epsilon_{11},...,\epsilon_{rd_{r}})$ acts on a weight
$(k_{11},...,k_{rd_{r}},w)$ by $(\epsilon_{11},...,\epsilon_{rd_{r}}).(k_{11},...,k_{rd_{r}},w)=(\epsilon_{11}k_{11},...,\epsilon_{rd_{r}}k_{rd_{r}},w)$. The length of $(\epsilon_{11},...,\epsilon_{rd_{r}})$ is $\#\left\{ (i,j)\mid\epsilon_{ij}=-1\right\} $.
$\rho$ denotes half the sum of the positive roots, which in our case
is $(1,...,1,0)$. The theorem assumes $V$ irreducible; we may treat
arbitrary semisimple representations by decomposing and taking direct
sums (of course this decomposition may not be unique in general).

Recall from above our representations
\[
\bigotimes_{i}Sym^{k_{i}-2d_{i}}(Sd_{i})
\]
where the $k_{i}$ are integers such that $k_{i}\geq2$, and that
\[
Sym^{k_{i}-2d_{i}}(Sd_{i})=\bigoplus_{(k_{i1},...,k_{id_{i}},a_{i1},...,a_{id_{i}})}\left(\bigotimes_{j} Sym^{k_{ij}-2-2a_{ij}}(Sd_{ij})\right)\otimes {\rm det}^{\sum a_{ij}}
\]
and hence 
\[
\bigotimes_{i}Sym^{k_{i}-2d_{i}}(Sd_{i})=\bigoplus_{(k_{11},...,k_{rd_{r}},a_{11},...,a_{rd_{r}})}\left(\bigotimes_{i,j} Sym^{k_{ij}-2-2a_{ij}}(Sd_{ij})\right)\otimes {\rm det}^{\sum a_{ij}}
\]
with $\left(\bigotimes_{i,j}Sym^{(k_{ij}-2)-2a_{ij}}(Sd_{ij})\right)\otimes{\rm det}^{\sum a_{ij}}$
irreducible of dominant weight 
\[
(k_{11}-2-2a_{11},...,k_{rd_{r}}-2-2a_{rd_{r}},k-2d)
\]
The BGG complex of $\left(\bigotimes_{i,j}Sym^{(k_{ij}-2)-2a_{ij}}(Sd_{ij})\right)\otimes{\rm det}^{\sum a_{ij}}$
therefore has $r$-th term 
\[
\bigoplus_{(\epsilon_{11},...,\epsilon_{rd_{r}})}U(\mathfrak{g}^{\star})\otimes_{U(\mathfrak{b}^{\star})}\chi(\epsilon_{11}(k_{11}-1-2a_{11})-1,...,\epsilon_{rd_{r}}(k_{rd_{r}}-1-2a_{rd_{r}})-1,k-2d)
\]
where the direct sum is taken over all $(\epsilon_{11},...,\epsilon_{rd_{r}})\in\left(W^{(r)}\right)^{d}$.
Note that $\epsilon_{ij}(k_{ij}-1-2a_{ij})-1$ is $k_{ij}-2-2a_{ij}$ if
$\epsilon_{ij}=1$ and $-k_{ij}+2a_{ij}$ if $\epsilon_{ij}=-1$.

\subsection{Dual BGG complexes for $X$\label{sub:2.5 FBGG}}

The automorphic vector bundle construction produces, given the BGG complex
of an irreducible representation $V$, a complex of vector bundles
and differential operators which is a quasi-isomorphic direct summand
of the de Rham complex of the vector bundle with connection associated
to $V$ (see e.g. \cite{Fal}, \cite{ChFa} or \cite{LaPo}). Specialized
to our situation, the theorem is:

\begin{thm}
(\cite{Fal} Thm 3, \cite{ChFa}) \label{thm: Faltings's BGG}We have,
associated to the irreducible representation of dominant weight $\lambda=(k_{1},...,k_{d},w)$,
over $\overline{\mathbb{Q}}$, a complex 
\[
0\rightarrow\mathcal{K}_{\lambda}^{0}\rightarrow...\rightarrow\mathcal{K}_{\lambda}^{d}\rightarrow0
\]
called the dual BGG complex, with $\mathcal{K}_{\lambda}^{r}=\bigoplus_{w\in W^{(r)}}W(w(\lambda+\rho)-\rho)^{\vee}$
on $X$ where the maps are Hecke-equivariant differential operators,
which is a quasi-isomorphic direct summand of the de Rham complex
$V(\lambda)^{\vee}\otimes_{\mathcal{O}_{X}}\Omega_{X}^{\bullet}$
of $V(\lambda)^{\vee}$.
\end{thm}

Here, as earlier and as will be the case in the rest of the article,
$V(\lambda)=V(k_{11},...,k_{rd_{r}},w)$ denotes the vector bundle with
connection associated to $\left(\bigotimes_{i,j}Sym^{k_{ij}}(Sd_{ij})\right)\otimes{\rm det}^{\left(w-\sum k_{ij}\right)/2}$.
As in the previous section, we may of course consider arbitrary semisimple
representations by decomposing and taking direct sums. Thus we get BGG complexes of $Sym^{k-2d}\left(H_{dR}^{1}(A/X)\right)$ resp. $\bigotimes_{i}Sym^{k_{i}-2d_{i}}\left(H_{dR}^{1}(A/X)_{i}\right)$ (associated with $Sym^{k-2d}(Sd^{\vee})$ resp. $\bigotimes_{i}Sym^{k_{i}-2d_{i}}(Sd_{i}^{\vee})$)  that are direct summands of their respective de Rham complexes. Here we are using that the action of $\mathcal{O}_{B^{op}}\otimes \mathbb{Q}_{p}=M_{2}(F_{p})= \prod_{i}M_{2}(F_{\mathfrak{p}_{i}})$ on $H^{1}_{dR}(A/X)$ gives a decomposition
\[
H^{1}_{dR}(A/X)= \bigoplus_{i}H^{1}_{dR}(A/X)_{i}
\]
We have
\[
BGG\left(\bigotimes_{i}Sym^{k_{i}-2d_{i}}(H^{1}_{dR}(A/X)_{i})\right)=
\]
\begin{equation}
= \bigoplus_{(k_{11},...,k_{rd_{r}},a_{11},...,a_{rd_{r}})}BGG\left(V(k_{11}-2a_{11}-2,...,k_{rd_{r}}-2a_{rd_{r}}-2,k-2d)^{\vee}\right) \label{eq: BGG2}
\end{equation}

and finally we note that the $r$-th term of the BGG complex of $V(k_{11},...,k_{rd_{r}},k-2d)^{\vee}(-\sum a_{ij})$ is 
\begin{equation}
\bigoplus_{(\epsilon_{11},...,\epsilon_{rd_{r}},w)}W(\epsilon_{11}(k_{11}-2a_{11}-1)-1,..., \epsilon_{rd_{r}}(k_{rd_{r}}-2a_{rd_{r}}-1)-1,k-2d) ^{\vee} \label{eq: BGG3}
\end{equation}

\section{Rigid and overconvergent de Rham cohomology \label{sec:3}}

As references for rigid cohomology we will mainly use \cite{LeSt1},
but see also (for example) the papers \cite{Ked1}, \cite{Ked2} for
a slightly different and perhaps more concrete perspective, or the
paper \cite{LeSt2} for a site-theoretic framework paralleling that
of crystalline cohomology. We are ultimately interested in the rigid
cohomology groups of $X_{\mathbb{F}_{p}}^{ord}$ (and the overconvergent
de Rham cohomology groups of $X_{rig}^{ord}$) with values in certain
overconvergent $F$-isocrystals (or overconvergent differential modules),
considered as Hecke modules and as $F$-isocrystals. Before we proceed, let us recall the notion of a frame from \cite{LeSt1} (Def. 3.1.5). 
\begin{defn}
Let $K$ be a complete valued field, let $\mathcal{V}$ be its valuation
ring, and $k$ its residue field. A ($K$-)frame is a diagram 
\[
S\hookrightarrow T\hookrightarrow P
\]
 consisting of an open immersion of $k$-schemes $S\hookrightarrow T$
and a closed immersion of the $k$-scheme $T$ into a formal $\mathcal{V}$-scheme
$P$.
\end{defn}
We will also write frames as $S\subseteq T\subseteq P$. Frames will
be important later when we consider rigid cohomology. Morphisms of
frames are simply commutative diagrams 
\[
\xymatrix{S\ar@{^{(}->}[r]\ar[d]^{f} & T\ar@{^{(}->}[r]\ar[d]^{g} & P\ar[d]^{u}\\
S^{\prime}\ar@{^{(}->}[r] & T^{\prime}\ar@{^{(}->}[r] & P^{\prime}
}
\]
 where $f$ and $g$ are morphisms of $k$-schemes and $u$ is a morphism
of formal $\mathcal{V}$-schemes (\cite{LeSt1} Def. 3.1.6). The morphism
is said to be quasi-compact if $u$ is quasi-compact (\cite{LeSt1}
Def. 3.2.1), and etale (resp. smooth) if $u$ is etale (resp. smooth)
in a neighbourhood of $S$ (inside $P$) (\cite{LeSt1} Def. 3.3.5).
The morphism is said to be proper if $g$ is proper (\cite{LeSt1}
Def. 3.3.10). Given a morphism of frames as above, $u$ induces a morphism
$u_{K}\,:\, P_{rig}\rightarrow P_{rig}^{\prime}$ of rigid analytic
varieties which maps $]S[_{P}$ into $]S^{\prime}[_{P^{\prime}}$.

To analyze the rigid cohomology groups of certain overconvergent isocrystals on $X_{\mathbb{F}_{p}}^{ord}$ we introduce the frames 
\[
A_{\mathbb{F}_{p}}^{ord}\subseteq A_{\mathbb{F}_{p}}^{DP}\subseteq\hat{\mathcal{A}}^{DP}
\]
\[
X_{\mathbb{F}_{p}}^{ord}\subseteq X_{\mathbb{F}_{p}}^{DP}\subseteq\hat{\mathcal{X}}^{DP}
\]
\[
A_{k_{L}}^{PR}=A_{k_{L}}^{PR}\subseteq \hat{\mathcal{A}}^{PR}
\]
\[
X_{k_{L}}^{PR}=X_{k_{L}}^{PR}\subseteq\hat{\mathcal{X}}^{PR}
\]
\[
X_{k_{L}}^{ss,PR}=X_{k_{L}}^{ss,PR} \subseteq\hat{\mathcal{X}}^{PR}
\]

Note that there is a Cartesian map of frames from the first frame above to
the second (this gives the definition of $A_{\mathbb{F}_{p}}^{ord}$) resp. from the third to the fourth coming from the map $\mathcal{A}^{DP}\rightarrow\mathcal{X}^{DP}$ resp. $\mathcal{A}^{PR}\rightarrow \mathcal{X}^{PR}$.
We may use these frames to interpret overconvergent isocrystals (and
rigid cohomology) on $X_{\mathbb{F}_{p}}^{ord}$ resp. $X_{k_{L}}^{PR}$ as overconvergent (on $X_{rig}^{ord}$ resp. $X_{rig}$) differential modules (and de Rham cohomology)
on $X_{rig}$, since $X_{\mathbb{F}_{p}}^{DP}$ and $X_{k_{L}}^{PR}$
are proper and $\hat{\mathcal{X}}^{DP}$ resp. $\hat{\mathcal{X}}^{PR}$
are smooth in a neighbourhood of $X_{\mathbb{F}_{p}}^{ord}$ resp. $X_{k_{L}}^{PR}$ (this is Cor. 8.1.9 and Prop. 7.2.13 of \cite{LeSt1}).
We may also use lifts of Frobenius to calculate Frobenius actions
(see \cite{LeSt1} \S 8.3). It should be noted that functoriality is
not as rigid as frames look like; given a frame $X\subseteq Y\subseteq P$
one does not need to lift morphisms to $P$, it is sufficient to lift
them to a strict neighbourhood of $]X[_{P}$ (tube of $X$ inside
$P$), see \cite{LeSt1} Prop. 8.1.6 (see also \cite{LeSt2}, where
this observation is built into the foundations).

Consider the universal abelian varieties $A_{\mathbb{F}_{p}}^{ord}\rightarrow X_{\mathbb{F}_{p}}^{ord}$ resp. $A_{k_{L}}^{PR}\rightarrow X_{k_{L}}^{PR}$. The relative rigid cohomology groups $H_{rig}^{1}(A_{\mathbb{F}_{p}}^{ord}/X_{\mathbb{F}_{p}}^{ord})$ resp. $H^{1}_{rig}(A_{k_{L}}^{PR}/X^{PR}_{k_{L}})$
are overconvergent $F$-isocrystals on $X_{\mathbb{F}_{p}}^{ord}$ resp. $X_{k_{L}}^{PR}$ and its fibres over closed points are the contravariant Dieudonne module of the corresponding fibre of the universal abelian variety with its Frobenius action (see e.g. \cite{Tzu} Thm 4.1.4 for the relevant base change assertion).
The morphism from the Frobenius pullback of $H_{rig}^{1}(A_{\mathbb{F}_{p}}^{ord}/X_{\mathbb{F}_{p}}^{ord})$
to $H_{rig}^{1}(A_{\mathbb{F}_{p}}^{ord}/X_{\mathbb{F}_{p}}^{ord})$ is given
by pull back along the relative Frobenius of $A_{\mathbb{F}_{p}}^{ord}/X_{\mathbb{F}_{p}}^{ord}$ (this is the induced
Frobenius structure on rigid cohomology; see also the remark by the
end of section 2 of \cite{Col}). By \cite{Tzu} Thm 4.1.4 again, the restrictions
of $H_{rig}^{1}(A_{k_{L}}^{PR}/X_{k_{L}}^{PR})$ to $X_{k_{L}}^{ord}$
resp. $X_{k_{L}}^{ss,PR}$ are $H_{rig}^{1}(A_{k_{L}}^{ord}/X_{k_{L}}^{ord})$
resp. $H_{rig}^{1}(A_{k_{L}}^{ss,PR}/X_{k_{L}}^{ss,PR})$
(where $A_{k_{L}}^{ss,PR}$ denotes the restriction of $A_{k_{L}}$ to $X_{k_{L}}^{ord}$
resp. $X_{k_{L}}^{ss,PR}$). Since rigid cohomology commutes finite base extensions we have $H_{rig}^{1}(A_{\mathbb{F}_{p}}^{ord}/X_{\mathbb{F}_{p}}^{ord})\otimes_{\mathbb{Q}_{p}}L=H_{rig}^{1}(A_{k_{L}}^{ord}/X_{k_{L}}^{ord})$ (as $F$-isocrystals). The rigid cohomology of certain summands of symmetric powers of these overconvergent $F$-isocrystals will be our main object of study in this section.

\begin{rem}
1)\label{rem:Hecke actions on coh} As we are using specific frames
to compute rigid cohomology we will think of these rigid cohomology
groups and overconvergent de Rham cohomology groups as ``the same'';
even though we write ``$H_{rig}$'' from now on we may occasionally
want to think of these as overconvergent de Rham cohomology groups.
Recall the Hecke algebras $\mathcal{H}_{K}$, $\mathcal{H}_{K}^{p}$,
$\mathcal{H}^{p}$ and $\mathcal{H}$ introduced by the end of section \ref{sub:U_p}.
$\mathcal{H}_{K}^{p}$ acts as correspondences on $\mathcal{X}^{PR}$ and $\mathcal{X}^{PR}$ and
both morphisms defining the correspondences are finite etale. Hence
we get compatible actions on $X_{\mathbb{F}_{p}}^{ord}$, $X^{PR}_{k_{L}}$ and $X_{rig}$
which preserve $X_{k_{L}}^{ord}$ and $X_{k_{L}}^{ss,PR}$ (as well as $X_{rig}^{ord}$ and $X_{rig}^{ss}$, by compatibility). At $p$ we will only consider $U_{p}$, $T_{p}$, $\left\langle p\right\rangle ^{\pm1}$
and $Fr$. $Fr$ is a Frobenius lift for $X_{\mathbb{F}_{p}}^{ord}$
and hence gives a concrete way of computing Frobenius actions on the
relevant overconvergent $F$-isocrystals on $X_{\mathbb{F}_{p}}^{ord}$.
Furthermore, $Fr$,$U_{p}$, $\left\langle p\right\rangle ^{\pm1}$
and $T_{p}$ define correspondences with both maps etale on $X_{rig}$
and $X_{rig}^{ord}$, and will act on the relevant cohomology groups
and spaces of automorphic forms on $X_{rig}$ and $X_{rig}^{ord}$.

2) There is a point of concern of what the natural choice of base
field is; when working with automorphic forms it is perhaps $\mathbb{C}_{p}$,
and $\mathbb{Q}_{p}$ or a finite extension therefore when working
with overconvergent $F$-isocrystals. In this section, when we consider schemes over $\mathbb{F}_{p}$, $k_{L}$ and $\overline{\mathbb{F}}_{p}$ respectively, our frames will be $\mathbb{Q}_{p}$-, $L$- and $\mathbb{C}_{p}$-frames respectively. We would therefore like to know
that our constructions commute with the change of base field from a
finite extension of $\mathbb{Q}_{p}$ to $\mathbb{C}_{p}$. Rigid
cohomology (and coherent cohomology) commutes with a finite extension
of base field (\cite{LeSt1} Proposition 8.2.14). However, rigid cohomology
is not known in general to commute with change of base field (we are
grateful to Le Stum for informing us of this). For us however we may
avoid this as follows. First note that for coherent cohomology of
complexes on affinoids this is clear, this is just flat base change
for modules (there is also no higher coherent cohomology) . As
overconvergent de Rham cohomology on $X^{ord}_{rig}$ is just the direct limit of de Rham
cohomology taken over a cofinal set of strict neighbourhoods (which
we may chose to be affinoid), the assertion follows by exactness of
direct limits of modules and the fact that direct limits commute with
tensor products. This shows that rigid cohomology on $X^{ord}_{\mathbb{F}_{p}}$ commutes with change of base field. For $X_{k_{L}}^{PR}$ we may use rigid analytic GAGA and flat base change in the algebraic category. Finally, the base change assertion for $X_{k_{L}}^{ss,PR}$ follows from that for $X_{k_{L}}^{ord}$ and $X_{k_{L}}^{PR}$ by the excision sequence and the functoriality of the base change morphism. Thus no real problem arises from changing base field.
We hope that the reader will find it easy to determine which base
field is appropriate throughout this section. The only point that
perhaps requires some clarification is that when the base field is
not $\mathbb{Q}_{p}$, the semilinear Frobenius action on overconvergent
$F$-isocrystals is a semilinearization of the linear Frobenius
action on automorphic forms (consider for example the upcoming Theorem
\ref{thm:rig coh as oc aut forms}). When the base field is $\mathbb{Q}_{p}$,
however, both actions agree, and since slopes for $F$-isocrystals
remain the same after change of base field, the linear Frobenius action
and the semilinear Frobenius action on the relevant rigid cohomology
groups will have the same slopes. This is the reason that we are using the Deligne-Pappas model; we wish to have a $\mathbb{Q}_{p}$-frame for the ordinary locus. These observations regarding base fields will be implicitly applied
when we compare slopes on rigid cohomology with $U_{p}$-slopes in
section \ref{sub:Classicality}. 
\end{rem}

\subsection{Relation to overconvergent automorphic forms \label{sub: 3.1 comparison}}

Given the Cartesian maps of frames 
\[
\left(A_{\mathbb{F}_{p}}^{ord}\subseteq A_{\mathbb{F}_{p}}^{DP}\subseteq \hat{\mathcal{A}}^{DP}\right) \rightarrow \left(X_{\mathbb{F}_{p}}^{ord}\subseteq X_{\mathbb{F}_{p}}^{DP}\subseteq\hat{\mathcal{X}}^{DP}\right)
\]
\[
\left(A_{k_{L}}^{PR}= A_{k_{L}}^{PR}\subseteq \hat{\mathcal{A}}^{PR}\right)\rightarrow \left(X_{k_{L}}^{PR}= X_{k_{L}}^{PR}\subseteq \hat{\mathcal{X}}^{PR}\right)
\]
and the fact that these both frames realize rigid cohomology, we deduce
from the definition of rigid cohomology that the overconvergent resp. convergent $F$-isocrystal
$H_{rig}^{1}(A_{\mathbb{F}_{p}}^{ord}/X_{\mathbb{F}_{p}}^{ord})$ resp. $H^{1}_{rig}(A_{k_{L}}^{PR}/X_{k_{L}}^{PR})$ is realized
by the overconvergent resp. convergent de Rham cohomology $H^{1}_{dR}(A_{rig}^{ord}\subseteq A_{rig}/X_{rig}^{ord}\subseteq X_{rig})$ resp. $H^{1}_{dR}(A_{rig}/X_{rig})$ (the former over $\mathbb{Q}_{p}$, the latter over $L$). Since $\mathcal{A}$ and $\mathcal{X}$ are proper we have $H^{1}_{dR}(A_{rig}/X_{rig})= H^{1}_{dR}(A_{an}/X_{an})$ and by comparison between algebraic and rigid analytic de Rham cohomology
(see e.g. \cite{AnBa} Thm. IV.4.1) we have $H^{1}_{dR}(A_{an}/X_{an})= H^{1}_{dR}(A/X)$,
hence $H^{1}_{dR}(A_{rig}/X_{rig})=\left(H_{dR}^{1}(A/X)\right)_{an}$,
and similarly for its symmetric powers.

To simplify notation we will put 
\[
V^{\dagger}(k_{11},...,k_{rd_{r}},w)=j_{X_{\overline{\mathbb{F}}_{p}}^{ord}}^{\dagger}\left(V(k_{11},...,k_{rd_{r}},w)_{an}\right)
\]
and
\[
W^{\dagger}(k_{11},...,k_{rd_{r}},w)=j_{X_{\overline{\mathbb{F}}_{p}}^{ord}}^{\dagger}\left(W(k_{11},...,k_{rd_{r}},w)_{an}\right)
\]
where $j$ denotes the open immersion $X^{ord}_{rig}\hookrightarrow X_{rig}$. These are overconvergent sheaves on $X_{rig}^{ord}$ (see \cite{LeSt1} section 5.1 for the definition of $j^{\dagger}$, it is probably easiest
to use his Prop. 5.1.12 as the definition). We may replace $\overline{\mathbb{F}}_{p}$
by $\mathbb{F}_{p}$ when the representation is defined over $\mathbb{Q}_{p}$.
Applying analytification and $j_{X_{\overline{\mathbb{F}}_{p}}^{ord}}^{\dagger}$ (both are
exact functors) to our dual BGG complexes, we get overconvergent dual
BGG complexes $\mathcal{K}_{(k_{11},...,k_{rd_{r}},w)}^{\dagger,\bullet}$
on $X_{rig}^{ord}$ which are direct summands of corresponding the
overconvergent de Rham complexes. Note that 
\[
H^{0,\dagger}(X_{rig}^{ord},W(k_{11},...,k_{rd_{r}},w))=H^{0}(X_{rig},W^{\dagger}(k_{11},...,k_{rd_{r}},w))
\]
so the $W^{\dagger}(k_{11},...,k_{rd_{r}},w)$ are the "sheaves of overconvergent automorphic forms". We now wish to interpret $H_{rig}^{d}(X_{\overline{\mathbb{F}}_{p}}^{ord},V^{\dagger}(k_{11},...,k_{rd_{r}},w)^{\vee})$ in terms of overconvergent automorphic forms. Since $X_{\mathbb{F}_{p}}^{ord}$
is affine, $X_{rig}^{ord}$ and its small strict neighbourhoods are
quasi-Stein and hence
\[
H^{i}(X_{rig},W^{\dagger}(k_{11},...,k_{rd_{r}},w))=0
\]
for $i\geq1$ (coherent cohomology). From this we get the following
theorem, which is the analogue of Theorem 5.4 of \cite{Col}:

\begin{thm}
\label{thm:rig coh as oc aut forms}$H_{rig}^{i}(X_{\overline{\mathbb{F}}_{p}}^{ord},V^{\dagger}(k_{11},...,k_{rd_{r}},w)^{\vee})$ is equal to
\[
h^{i}\left(\bigoplus_{(\epsilon_{j})\in W^{(\bullet)}}H^{0}(X_{rig}, W^{\dagger}(\epsilon_{11}(k_{11}+1)-1,...,\epsilon_{rd_{r}}(k_{rd_{r}}+1)-1,w)^{\vee})\right)
\]
 Here $h^{i}$ stands for ``$i$-th cohomology of the complex''.
In particular, if we denote by $\theta_{(k_{11},...,k_{rd_{r}},w)}$
the map 
\[
\bigoplus_{(\epsilon_{j})\in W^{(d-1)}}W^{\dagger}(\epsilon_{11}(k_{11}+1)-1,...,\epsilon_{rd_{r}}(k_{rd_{r}}+1)-1,w)^{\vee}\,\longrightarrow W^{\dagger}(k_{11}+2,...,k_{rd_{r}}+2,-w)
\]
 and by abuse of notation also the induced map
\[
\bigoplus_{(\epsilon_{j})\in W^{(d-1)}}H^{0}(X_{rig},W^{\dagger}(\epsilon_{11}(k_{11}+1)-1,...,\epsilon_{rd_{r}}(k_{rd_{r}}+1)-1,w)^{\vee})\,\longrightarrow 
\]
\[
\longrightarrow H^{0}(X_{rig},W^{\dagger}(k_{11}+2,...,k_{rd_{r}}+2,-w))
\]
 of global sections, then 
\[
H_{rig}^{d}(X_{\overline{\mathbb{F}}_{p}}^{ord},V^{\dagger}(k_{11},...,k_{rd_{r}},w)^{\vee})={\rm Coker}\,\theta_{(k_{11},..,k_{rd_{r}},w)}
\]
\end{thm}

\begin{proof}
We have 
\[
H_{rig}^{i}(X_{\overline{\mathbb{F}}_{p}}^{ord},V^{\dagger}(k_{11},...,k_{rd_{r}},w)^{\vee})=H_{dR}^{i}(X_{rig},V^{\dagger}(k_{11},...,k_{rd_{r}},w)^{\vee})=H^{i}(X_{rig},\mathcal{K}_{k_{11},...,k_{rd_{r}},w}^{\dagger,\bullet})
\]
where the first equality is by the definition of rigid cohomology and the second is by the
quasi-isomorphism of the de Rham complex of $V^{\dagger}(k_{11},...,k_{rd_{r}},w)^{\vee}$
and $\mathcal{K}_{k_{11},...,k_{rd_{r}},w}^{\dagger,\bullet}$. The vanishing
\[
H^{i}(X_{rig},W^{\dagger}(k_{11},...,k_{rd_{r}},w))=0
\]
for $i\geq1$ then gives the
first statement by the hypercohomology spectral sequence. The last
statement follows from the first and the definitions. \end{proof}

\begin{rem}
Note that all the previous equalities of cohomology groups are valid
as equalities of Hecke modules (cf. Remark \ref{rem:Hecke actions on coh}).
\end{rem}

Now look at $H_{rig}^{d}\left(X_{\mathbb{F}_{p}}^{ord},Sym^{k-2d}\left(H_{rig}^{1}(A_{\mathbb{F}_{p}}^{ord}/X_{\mathbb{F}_{p}}^{ord})\right)\right)$. It is an $F$-isocrystal over $\mathbb{Q}_{p}$. It has a direct summand 
\[
H^{d}_{rig}\left(X_{\mathbb{F}_{p}}^{ord},\bigotimes_{i}Sym^{k_{i}-2d_{i}}\left(H_{rig}^{1}(A_{\mathbb{F}_{p}}^{ord}/X_{\mathbb{F}_{p}}^{ord})_{i}\right)\right)
\]
We have
\[
\bigotimes_{i}Sym^{k_{i}-2d_{i}}\left(H_{rig}^{1} (A_{\overline{\mathbb{F}}_{p}}^{ord}/X_{\overline{\mathbb{F}}_{p}}^{ord})_{i}\right)=  
\]
\[
=\bigoplus_{(k_{11},...,k_{rd_{r}},a_{11},...,a_{rd_{r}})}V(k_{11}-2-2a_{11},...,k_{rd_{r}}-2-2a_{rd_{r}},k-2d)^{\vee}
\]
 and hence, letting $\mathcal{E}_{k_{1},...,k_{r}}= \bigotimes_{i}Sym^{k_{i}-2d_{i}}\left(H_{rig}^{1} (A_{\mathbb{F}_{p}}^{ord}/X_{\mathbb{F}_{p}}^{ord})_{i}\right)$,
\[
H_{rig}^{d}\left(X_{\overline{\mathbb{F}}_{p}}^{ord}, \mathcal{E}_{k_{1},...,k_{r}}\right)= \bigoplus_{(k_{11},..,k_{rd_{r}},a_{11},...,a_{rd_{r}})}{\rm Coker}\: \theta_{(k_{11}-2-2a_{11},...,k_{rd_{r}}-2-2a_{rd_{r}},k-2d)}
\]
Let us now fix $(k_{1},...,k_{r})$ and $(k_{11},...,k_{rd_{r}})$ such that $\sum k_{i}=k$, $\sum_{j} k_{ij}=k_{i}$ and $k_{ij}\geq 2$ for all $i$ and $j$. One of the summands above is
${\rm Coker}\:\theta_{(k_{11}-2,...,k_{rd_{r}}-2,k-2d)}$ which is a quotient
of $H^{0}(X_{rig},W^{\dagger}(k_{11},...,k_{rd_{r}},-k+2d))$. This is the
part of the cohomology we will be interested in.

\subsection{Small slope criterion for occurring in the cohomology\label{sub: 3.2 small slope in coh}}

Next, we need to know how to normalize the $U_{p}$-operator to achieve
optimal $p$-integrality. This has been done by Hida in \cite{Hid} in the general unramified situation and his method works for our $U_{p}$-operator as well, using the description as the trace of Frobenius (up to a diamond operator). We can formulate the result as:

\begin{prop}
\label{pro: scaling hecke}The $U_{p}$-operator is $p$-integral
on $H^{0}\left(X_{rig}^{ord},W\left(k_{11},...,k_{rd_{r}},-\left(\sum k_{ij}\right)+2d\right)\right)$
and hence on $H^{0}\left(X_{rig},W^{\dagger}\left(k_{11},...,k_{rd_{r}},-\left(\sum k_{ij}\right)+2d\right)\right)$
(in the sense that its eigenvalues are $p$-integral) and has slope
$0$-eigenvectors on both these spaces. Moreover, shifting the central
character up by $2$ scales $U_{p}$ by $p^{-1}$.\end{prop}

\begin{proof}
As mentioned before the statement of the proposition, the first part
follows by a standard calculation following Hida and the second part. We remark that this calculation is entirely analogous to the standard $q$-expansion calculation, using Serre-Tate coordinates instead of the Tate abelian variety. The proof of the second part is also by a standard calculation. Let us outline the argument. First,
we prove the analogue statement over the complexes. Let $\Gamma=G^{\star}(\mathbb{Q})\cap K$
and let $h=\left(\begin{array}{cc}
p\\
 & 1
\end{array}\right)$. Fix a weight $(k_{11},...,k_{rd_{r}},w)$ and write $\chi=\chi(k_{11},...,k_{rd_{r}},w)$.
We may interpret automorphic forms of level $K$ and weight $(k_{11},...,k_{rd_{r}},w)$
over $\mathbb{C}$ as functions 
\[
f\,:\, G^{\star}(\mathbb{R})\rightarrow\mathbb{C}
\]
satisfying $f(\gamma g)=f(g)$ and $f(gk)=\chi(k)^{-1}f(g)$ for $\gamma\in\Gamma$
and $k\in K_{\infty}$, or equivalently as functions 
\[
\phi\,:\, G^{\star}(\mathbb{A})\rightarrow\mathbb{C}
\]
such that $\phi(\gamma g)=\phi(g)$ and $\phi(gk)=\chi(k_{\infty})^{-1}\phi(g)$
for $\gamma\in G^{\star}(\mathbb{Q})$ and $k\in K$ (plus analytic
conditions that we will not need and therefore not go into). Given
$f$, the associated $\phi$ is defined by $\phi(g)=f(g_{\infty})$.
Note that we may describe local sections of $W(k_{11},...,k_{rd_{r}},w)$
on $X(\mathbb{C})$ by the same equations, restricting the domain
of $f$ to any open $U^{\prime}$ which is the pullback of some analytic
open $U$ under the natural map $G^{\star}(\mathbb{R})\rightarrow X(\mathbb{C})$.
The adelic operator $U_{p}=[KhK]$, which in the classical setting
becomes $[\Gamma h^{-1}\Gamma]$, acts as 
\[
\left(U_{p}f\right)(g)=\sum_{i}f(h^{-1}\gamma_{i}g)
\]
 for some (any) set $\gamma_{1},...,\gamma_{r}$ of coset representatives
of $(\Gamma\cap h\Gamma h^{-1})\backslash\Gamma$. Now consider changing
the weight by a factor of ${\rm det}$, i.e. $(k_{11},...,k_{rd_{r}},w)$
goes to $(k_{11},...,k_{rd_{r}},w+2)$. There is an isomorphism of coherent
sheaves 
\[
\varphi\,:\, W(k_{11},...,k_{rd_{r}},w)\rightarrow W(k_{11},...,k_{rd_{r}},w+2)
\]
(which is valid over the reflex field) defined on local sections by
\[
(\varphi(f))(g)={\rm det}(g)^{-1}f(g)
\]
 Thus we see that 
\[
\left(U_{p}(\varphi(f))\right)(g)={\rm det}(h^{-1})^{-1}{\rm det}(g)^{-1}\sum f(h\gamma_{i}g)=p.\left(\varphi\left(U_{p}f\right)\right)(g)
\]
 which is the result we wanted. Now as this identity holds analytically
over $\mathbb{C}$, it also holds formally around every $\mathbb{C}$-point,
hence formally around every $\overline{\mathbb{Q}}$-point, and hence
rigid analytically in the ordinary locus by the principle of analytic
continuation (the ordinary locus is connected, and contains $\overline{\mathbb{Q}}$-points). \end{proof}

\begin{rem}
\label{rem:after scaling of U_p}1) The choice $\phi(g)=f(g_{\infty})$
is nonstandard (but seems to the author to be a fairly natural choice).
This is what forces $U_{p}$ to become $[\Gamma h^{-1}\Gamma]$ in
the classical setting; it differs from the usual choice using $\left(\begin{array}{cc}
1\\
 & p
\end{array}\right)$ by a central factor of $\left(\begin{array}{cc}
p\\
 & p
\end{array}\right)$, which reflects the fact that we didn't throw in determinant factors
in the equivalence $f\leftrightarrow\phi$.

2) We will also define the action of $Fr$ on $H^{0}\left(X_{rig},W^{\dagger}\left(k_{11},...,k_{rd_{r}},w\right)\right)$
by using the previously defined action on $H^{0}\left(X_{rig},W^{\dagger}\left(k_{11},...,k_{rd_{r}},-\left(\sum k_{ij}\right)+2d\right)\right)$
and declaring that shifting the central character up by $2$ scales
$Fr$ by $p^{-1}$. This corresponds to the interpretation of the
automorphic vector bundle of ${\rm det}$ as the Tate twist $\mathbb{Q}_{p}(1)$.
\end{rem}

We may now prove the analogue of Lemma 6.3 of \cite{Col}.
\begin{cor}
\label{cor:not in the image of theta}Let $k_{i}\geq2$ for all $i$.
If $f\in H^{0}\left(X_{rig},W^{\dagger}\left(k_{11},...,k_{rd_{r}},-\left(\sum k_{ij}\right)+2d\right)\right)$ is a $U_{p}$-eigenform
of slope less than $\inf_{i,j}(k_{ij}-1)$, then $f$ is not in the image
of $\theta$.\end{cor}

\begin{proof}
Recall that $\theta$ is a $U_{p}$-equivariant map 
\[
\bigoplus_{i,j}H^{0}\left(X_{rig},W^{\dagger}\left(k_{11},...,2-k_{ij},...,k_{rd_{r}},-\left(\sum k_{ij}\right)+2d\right)\right)\,\longrightarrow 
\]
\[
\longrightarrow H^{0}\left(X_{rig},W^{\dagger}\left(k_{11},...,k_{rd_{r}},-\left(\sum k_{ij}\right)+2d\right)\right)
\]
Here the right hand side has the optimal $U_{p}$ whereas, by the previous Proposition, the optimal
$U_{p}$ for weight $(k_{11},...,2-k_{ij},...,k_{rd_{r}})$ occurs with central
character 
\[
-\left(2-k_{ij}+\sum_{(i^{\prime},j^{\prime})\neq (i,j)}k_{i^{\prime}j^{\prime}}\right)+2d=\left(-\left(\sum k_{i^{\prime}j^{\prime}}\right)+2d\right)+2(k_{ij}-1)
\]
Thus $U_{p}$ acting on $H^{0}\left(X_{rig},W^{\dagger}\left(k_{11},...,2-k_{ij},...,k_{rd_{r}},-\left(\sum k_{ij}\right)+2d\right)\right)$
has eigenvalues of valuation $\geq k_{ij}-1$ by the previous Proposition.
This proves the Corollary.
\end{proof}

Thus, again for fixed $(k_{11},...,k_{rd_{r}})$ with $\sum_{j} k_{ij}=k_{i}$, $\sum k_{i}=k$, $H_{rig}^{d}\left(X_{\mathbb{F}_{p}}^{ord}, \mathcal{E}_{k_{1},...,k_{r}}\right)$
has a sub-Hecke module consisting of the overconvergent automorphic
forms of weight $(k_{11},...,k_{rd_{r}},-k+2d)$ of $U_{p}$-slope $<\inf_{i,j}(k_{ij}-1)$.

\subsection{\label{sub:The-excision-sequence}The excision sequence and a slope
criterion \label{sub:3.3 excision}}

The next thing to do is to is to start analyzing $H_{rig}^{d}\left(X_{\mathbb{F}_{p}}^{ord},\mathcal{E}_{k_{1},...,k_{r}}\right)$ using the formalism of rigid cohomology. To simplify notation we will write
$\underline{k}$ for $(k_{1},...,k_{r})$ and we continue to assume $k_{ij}\geq2$ for all $i,j$. The excision
sequence in rigid cohomology gives us a Frobenius-equivariant exact
sequence
\[
...\rightarrow H_{rig}^{d}\left(X_{k_{L}}^{PR},\mathcal{E}_{\underline{k}}\right)\longrightarrow H_{rig}^{d}\left(X_{k_{L}}^{ord},\mathcal{E}_{\underline{k}}\right)\longrightarrow H_{X_{k_{L}}^{ss,PR},rig}^{d+1}\left(X_{k_{L}}^{PR},\mathcal{E}_{\underline{k}}\right)\rightarrow...
\]
Here we have some knowledge of $H_{rig}^{d}\left(X_{k_{L}},\mathcal{E}_{\underline{k}}\right)$
as a Hecke module from comparison theorems and ``classical'' automorphic
methods (Matsushima's formula). The problematic term is the contribution
from $H_{X_{k_{L}}^{ss,PR},rig}^{d+1}\left(X_{k_{L}},\mathcal{E}_{\underline{k}}\right)$.
We will deal with it by bounding its slopes. Before we do this we
simplify it somewhat as follows:

\begin{prop}
There is an isomorphism $H_{X_{k_{L}}^{ss,PR},rig}^{d+1}\left(X_{k_{L}},\mathcal{E}_{\underline{k}}\right)\cong H_{rig}^{d-1}\left(X_{k_{L}}^{ss,PR},\mathcal{E}_{\underline{k}}^{\vee}(d)\right)^{\vee}$ which is Hecke and Frobenius-equivariant (where $(d)$ denotes a Tate
twist by $d$). \end{prop}

\begin{proof}
This is Poincare duality, see \cite{Ked1} Thm 1.2.3 and also
\cite{Ked2} section 2.1 or \cite{LeSt1} Corollary 8.3.14 for the Frobenius-equivariant
formulation. Hecke equivariance follows since the Hecke action is
by correspondences.
\end{proof}

We want to bound the range of the slopes of $H_{rig}^{d-1}\left(X_{k_{L}}^{ss,PR},\mathcal{E}_{\underline{k}}^{\vee}(d)\right)^{\vee}$. To do this we will use \S 6.7 of \cite{Ked2}. Since the fibre of $H_{rig}^{1}(A_{k_{L}}^{PR}/X_{k_{L}}^{PR})$
at a closed point $x$ of $X_{k_{L}}$ is simply the rational Dieudonne
module of $A_{k_{L},x}$, we note that the slopes of $H_{rig}^{1}(A_{k_{L}}/X_{k_{L}})$
lie in $[0,1]$ (the definition is in the second paragraph of section 6.7
of \cite{Ked2}; these slopes are ``pointwise slopes''). However,
more importantly for us:

\begin{prop}
\label{pro:positive slope outside ord loc} The slopes
of $\mathcal{E}_{\underline{k}}$ on $X_{k_{L}}^{ss,PR}$
are in $[\lambda,k-2d-\lambda]$, where
\[
\lambda={\rm{inf}}_{i}\left((k_{i}-2d_{i}){\rm{inf}}(1/2,1/d_{i})\right).
\]
Note that $\lambda$ depends on $(k_{1},...,k_{r})$.
\end{prop}

\begin{proof}
Given a closed point $x$ of $X_{k_{L}}^{ss,PR}$, the corresponding abelian
variety $A_{x}$ is isogenous over $\overline{\mathbb{F}}_{p}$ to
the square of a non-ordinary abelian variety $A^{\prime}$ with real
multiplication by $F$ by the proof of Proposition 5.2 of \cite{Mil79} (the result as stated in \cite{Mil79} requires $F$ unramified at $p$, but this is not used in the proof of the particular fact we need). We may decompose the rational Dieudonne modules $D(A)$ and $D(A^{\prime})$ according to primes above $p$ in $F$:
\[
D(A)=\bigoplus_{i} D(A)_{i}
\]
\[
D(A^{\prime})=\bigoplus_{i} D(A^{\prime})_{i}
\]
We have $D(A)_{i}$=$D(A^{\prime})_{i}^{\oplus 2}$. Each $D(A^{\prime})_{i}$ is a rank $2$ rational Dieudonne module over $F_{\mathfrak{p}_{i}}$, coming from a rank $2$ Dieudonne module over $\mathcal{O}_{F_{\mathfrak{p}_{i}}}$. The slopes of those may be calculated by combing Theorem 5.2.1 \cite{GoOo}, which does the unramified case, and Theorem 9.2 of \cite{AnGo}, which does the totally ramified case (strictly speaking one should perhaps combine their methods, but this only amounts to changing notation in the proofs). The outcome is that $D(A^{\prime})_{i}$ (and hence $D(A)_{i}$) has either two slopes $a/d_{i}$ and $(d_{i}-a)/d_{i}$ (with $a\in \mathbb{Z}$, $0\leq a \leq d_{i}$) or a single slope $1/2$. If $A$, and hence $A^{\prime}$, is non-ordinary, then there exists an $i$ such that the slopes of $D(A^{\prime})_{i}$ are not $0$ and $1$. Define $\lambda (i)={\rm{inf}}(1/2,1/d_{i})$, then the slopes of $D(A^{\prime})_{i}$ are in the interval $[\lambda(i),1-\lambda(i)]$, and hence the slopes of
\[
\bigotimes_{i}Sym^{k_{i}-2d_{i}}D(A)_{i} 
\]
are in $[(k_{i}-2d_{i})\lambda(i),k-2d-(k_{i}-2d_{i})\lambda(i)]$ (since slopes behave additively with respect to tensor operations). Thus we see that the slopes of $\mathcal{E}_{\underline{k}}$ on $X^{ss,PR}_{k_{L}}$ are in $[\lambda,k-2d-\lambda]$, where $\lambda={\rm{inf}} _{i}\left((k_{i}-2d_{i})\lambda(i)\right)$, as desired. \end{proof}

Using this, we are ready to prove the main result of this section.
Recall that a Tate twist by $1$ decreases slopes by $1$ and that dualizing sends a slope to its negative.

\begin{thm}
The slopes of $H_{rig}^{d-1}\left(X_{k_{L}}^{ss,PR},\mathcal{E}_{\underline{k}}^{\vee}(d)\right)^{\vee}$
lie in $\left[\lambda +1,k-d-\lambda\right]$.\end{thm}

\begin{proof}
By the previous Proposition the slopes of $\mathcal{E}_{\underline{k}}$ on $X_{k_{L}}^{ss,PR}$
are in $\left[\lambda,k-2d-\lambda\right]$, so by
the remarks before this Theorem the slopes of $\mathcal{E}_{\underline{k}}^{\vee}(d)$
are in $\left[\lambda +d-k,-\lambda -d\right]$.

Next we apply Theorem 6.7.1 of \cite{Ked2}, a special case of which
says that if $S$ is a proper separated scheme of finite type over
$k_{L}$ of pure dimension $d-1$ and $\mathcal{F}$ is an
overconvergent $F$-isocrystal on $S$ with slopes in $[r,s]$, then
the slopes of $H_{rig}^{d-1}(S,\mathcal{F})$ are in $[r,s+d-1]$.
In our situation this allows us conclude that the slopes of $H_{rig}^{d-1}\left(X_{k_{L}}^{ss,PR},\mathcal{E}_{\underline{k}}^{\vee}(d)\right)$
are in $\left[\lambda +d-k,-\lambda -1\right]$.
Dualizing we see that the slopes of $H_{rig}^{d-1}\left(X_{k_{L}}^{ss,PR},\mathcal{E}_{\underline{k}}^{\vee}(d)\right)^{\vee}$
lie in $\left[\lambda +1,k-d-\lambda\right]$ as
desired. \end{proof}

\begin{cor}
\label{cor:small slope Fr}The slopes of \textup{$H_{rig}^{d}\left(X_{k_{L}}^{ord},\mathcal{E}_{\underline{k}}\right)$
lie in $[0,k-d]$. Thus the part of cohomology with slopes in $\left[0,\lambda +1\right)\cup\left(k-d-\lambda,k-d\right]$
lies in the image of $H_{rig}^{d}\left(X_{k_{L}}^{PR},\mathcal{E}_{\underline{k}}\right)$.}\end{cor}

\begin{proof}
That the slopes of $H_{rig}^{d}\left(X_{k_{L}}^{ord},\mathcal{E}_{\underline{k}}\right)$
lie in $[0,k-d]$ follows from noting that the slopes of $\mathcal{E}_{\underline{k}}$
are in $[0,k-2d]$ (since the slopes of $H_{rig}^{1}(A_{k_{L}}^{PR}/X_{k_{L}}^{PR})$
are in $[0,1]$) and applying Theorem 6.7.1 of \cite{Ked2} (not the
same special case as before, but the same if you replace ``proper''
by ``smooth''). The second part then follows by the Theorem and
the excision sequence, as the part of cohomology with slopes in $\left[0,\lambda +1\right)\cup\left(k-d-\lambda,k-d\right]$
necessarily gets killed when mapped to $H_{X_{k_{L}}^{ss,PR},rig}^{d+1}\left(X_{k_{L}},\mathcal{E}_{\underline{k}}\right)$
and hence lies in the image of $H_{rig}^{d}\left(X_{k_{L}},\mathcal{E}_{\underline{k}}\right)$.
\end{proof}

\subsection{\label{sub:Classicality}Classicality for forms of small slope, the
case of arbitrary $d$ \label{sub: 3.4 main result d>1}}

Throughout this section we encourage the reader to keep part 2) of
Remark \ref{rem:Hecke actions on coh} in mind. Recall the Frobenius
correspondence $Fr$ on $X_{rig}^{ord}$ that we defined in section \ref{sub:U_p},
and that it overconverges. Composing $Fr$ with $U_{p}$ in one way
gives the correspondence 
\[
r=(r_{1},r_{2})\,:\, X_{rig}^{ord}\rightarrow X_{rig}^{ord}\times X_{rig}^{ord}
\]
with $r_{1}=Fr$, $r_{2}=\left\langle p\right\rangle ^{-1}Fr$ (we
define composition of correspondences as in \cite{Lau1} section 1.6). As $\left\langle p\right\rangle ^{\pm1}$ commutes with the
Frobenius morphism we rewrite this correspondence as the composition
of $\left\langle p\right\rangle $ with the correspondence 
\[
r^{\prime}=(r_{1}^{\prime},r_{2}^{\prime})\,:\, X_{rig}^{ord}\rightarrow X_{rig}^{ord}\times X_{rig}^{ord}
\]
with $r_{1}^{\prime}=r_{2}^{\prime}=Fr$. Transferring differentials
as for $Fr$ and $U_{p}$ we deduce that the action of $r^{\prime}$ on $H^{0}\left(X_{rig},W^{\dagger}\left(k_{11},...,k_{rd_{r}},-\left(\sum k_{ij}\right)\right)\right)$
is by $p^{k+d}$ ($p^{k}$ comes from the transfer of differentials, $p^{d}$ is the degree of the morphism $Fr$), and hence acts
on $H^{0}\left(X_{rig},W^{\dagger}\left(k_{11},...,k_{rd_{r}},-\left(\sum k_{ij}\right)+2d\right)\right)$
as $p^{k-d}$ (by Prop. \ref{pro: scaling hecke} and Rem. \ref{rem:after scaling of U_p}).
Hence it acts on $H_{rig}^{d}\left(X_{\mathbb{F}_{p}}^{ord},\mathcal{E}_{\underline{k}}\right)$
by $p^{k-d}$, and therefore $r$ acts by $\left\langle p\right\rangle p^{k-d}$.
Since $H_{rig}^{d}\left(X_{\mathbb{F}_{p}}^{ord},\mathcal{E}_{\underline{k}}\right)$
is finite-dimensional, one-sided inverses are two-sided inverses and
we can conclude that 
\begin{equation}
Fr\circ U_{p}=U_{p}\circ Fr=\left\langle p\right\rangle p^{k-d}\label{eq:FR}
\end{equation}
on $H_{rig}^{d}\left(X_{\mathbb{F}_{p}}^{ord},\mathcal{E}_{\underline{k}}\right)$.
We may conclude that the slopes of $U_{p}$ acting on $H_{rig}^{d}\left(X_{\mathbb{F}_{p}}^{ord},\mathcal{E}_{\underline{k}}\right)$
lie in $[0,k-d]$ (as the eigenvalues of $\left\langle p\right\rangle $
are roots of unity), and we immediately deduce the following Lemma from Corollary \ref{cor:small slope Fr}:

\begin{lem}
\label{lem: slope}The part of $H_{rig}^{d}\left(X_{k_{L}}^{ord},\mathcal{E}_{\underline{k}}\right)$
with $U_{p}$-slope in \textup{$\left[0,\lambda\right)\cup\left(k-d-\lambda -1,k-d\right]$
is in the image of $H_{rig}^{d}\left(X_{k_{L}}^{PR},\mathcal{E}_{\underline{k}}\right)$.}
\end{lem}

From this, our classicality criterion follows. Let us first state
the following simple consequence of Matsushima's formula:

\begin{lem}
\label{lem:Matsushima}The Hecke module $H_{rig}^{d}\left(X_{\overline{\mathbb{F}}_{p}}^{PR},\mathcal{E}_{\underline{k}}\right)$
decomposes as a direct sum of Hecke modules of $K$-fixed vectors
associated to automorphic representations of $G^{\star}$.\end{lem}

\begin{proof}
The direct sum decomposition of $H_{rig}^{d}\left(X_{\overline{\mathbb{F}}_{p}}^{PR},\mathcal{E}_{\underline{k}}\right)$
reduces the question to the same assertion for the $H_{rig}^{d}(X_{\overline{\mathbb{F}}_{p}}^{PR},V^{\dagger}(k_{11},...,k_{rd_{r}},w)^{\vee})$.
Since our Hecke operators are defined over $\mathbb{Q}$, by a sequence
of comparison theorems/definitions (definition of rigid cohomology,
complex and rigid analytic/algebraic comparison of de Rham cohomology
and flat base change) we see that the Hecke modules $H_{rig}^{d}(X_{\overline{\mathbb{F}}_{p}}^{PR},V^{\dagger}(k_{11},...,k_{rd_{r}},w)^{\vee})$
and $H_{dR}^{d}(X(\mathbb{C}),V(k_{11},...,k_{rd_{r}},w)^{\vee})$ arise
as base changes of the same Hecke module over $\mathbb{Q}$. We have Matsushima's formula 
\[
H_{dR}^{d}(X(\mathbb{C}),V(k_{11},...,k_{rd_{r}},w)^{\vee})=\bigoplus_{\pi}m(\pi)\pi_{f}^{K}\otimes H^{d}(\mathfrak{g}^{\star},K_{\infty};\pi_{\infty}\otimes\xi(k_{11},...,k_{rd_{r}},w)^{\vee})
\]
(the standard reference is \cite{BoWa} VII.5.2, see Thm 3.2 of \cite{Yosh}
for the formulation above and some more details) where the summation
is over all irreducible admissible representations of $G^{\star}(\mathbb{A})$,
$m(\pi)$ is the multiplicity of $\pi$ in the appropriate summand
of $L^{2}(G^{\star}(\mathbb{Q})\backslash G^{\star}(\mathbb{A})^{1})$,
$\pi_{f}^{K}$ is the $K$-fixed vectors of the finite part $\pi_{f}$
of $\pi$, $H^{d}(\mathfrak{g}^{\star},K_{\infty};-)$ is $(\mathfrak{g}^{\star},K_{\infty})$-cohomology
with trivial Hecke action and $\xi(k_{11},...,k_{rd_{r}},w)=\left(\bigotimes_{i,j}Sym^{k_{ij}}(Sd_{ij})\right)\otimes det^{\left(w-\sum k_{ij}\right)/2}$.
As $m(\pi).dim\, H^{d}(\mathfrak{g}^{\star},K_{\infty};\pi_{\infty}\otimes\xi(k_{1},...,k_{d},w)^{\vee}=0$
unless $\pi$ is the automorphic representation associated to some
automorphic form of level $K$ and weight $(k_{11},...,k_{rd_{r}},-w)$,
the lemma follows. \end{proof}

\begin{thm}
\label{thm:classicality} a) Let $f$ be an overconvergent
Hecke eigenform for $\mathcal{H}$, of weight $(k_{11},...,k_{rd_{r}})$,
character $\chi$ for the diamond operators and with $U_{p}$-slope
in $\left[0,\lambda\right)\cup\left(k-d-\lambda -1,k-d\right]$,
and assume that it is not in the image of $\theta$. Then its system
of Hecke eigenvalues for $\mathcal{H}$ comes from the $p$-stabilization
of a classical form of level $K$.

b) Assume that $f$ is an overconvergent Hecke eigenform for $\mathcal{H}$,
of weight $(k_{11},...,k_{rd_{r}})$, character $\chi$ for the diamond
operators and with $U_{p}$-slope less than $\inf\left(k_{ij}-1,\lambda\right)$.
Then its system of Hecke eigenvalues for $\mathcal{H}$ comes from
the $p$-stabilization of a classical form of level $K$.\end{thm}

\begin{proof}
We look here at the direct summand $coker\:\theta_{(k_{11}-2,...,k_{rd_{r}}-2,k-2d)}$
of $H_{rig}^{d}\left(X_{\overline{\mathbb{F}}_{p}}^{ord},\mathcal{E}_{\underline{k}}\right)$.
By Corollary \ref{cor:not in the image of theta} part b) follows
directly from a), so we may focus on a). We assume that $f$ is not
in the image of $\theta$, hence its system of Hecke eigenvalues outside
$p$ occurs in $H_{rig}^{d}\left(X_{\overline{\mathbb{F}}_{p}}^{ord},\mathcal{E}_{\underline{k}}\right)$,
and by Lemma \ref{lem: slope} it comes from $H_{rig}^{d}\left(X_{\overline{\mathbb{F}}_{p}}^{PR},\mathcal{E}_{\underline{k}}\right)$.
Lemma \ref{lem:Matsushima} now gives the theorem for $\mathcal{H}^{p}$.
For $U_{p}$, note that the class of $f$ in $H_{rig}^{d}\left(X_{\overline{\mathbb{F}}_{p}}^{ord},\mathcal{E}_{\underline{k}}\right)$
is also an eigenvector for $Fr$ (by equation \ref{eq:FR}), hence
for $T_{p}$ as $T_{p}=U_{p}+Fr$. Since $H_{rig}^{d}\left(X_{\overline{\mathbb{F}}_{p}}^{PR},\mathcal{E}_{\underline{k}}\right)\rightarrow H_{rig}^{d}\left(X_{\overline{\mathbb{F}}_{p}}^{ord},\mathcal{E}_{\underline{k}}\right)$
is equivariant for $T_{p}$, it follows that the $T_{p}$-eigenvalue
of the class of $f$ is the $T_{p}$-eigenvalue of the associated
classical form $g$ of level $K$, and that its $U_{p}$-eigenvalue
satisfies the $p$-Hecke polynomial of $g$, as $U_{p}$ satisfies
$x^{2}-T_{p}x+\chi(p)p^{k-d}$. Hence the $U_{p}$-eigenvalue of $f$
agrees with that of a $p$-stabilization of $g$, which was what we
wanted to prove.
\end{proof}

\subsection{The case $d=1$; the Hecke modules \textmd{$H_{rig}^{1}(X_{\overline{\mathbb{F}}_{p}}^{ord},V^{\dagger}(k-2,k-2)^{\vee})$
\label{sub:3.5 d=00003D1}}}

For completeness we give a separate treatment of the case $d=1$ in this subsection, where we
can obtain better results by methods similar to those in \cite{Col}. We will drop the superscripts $^{PR}$ and $^{DP}$ since we are in an unramified case and the Pappas-Rapoport and Deligne-Pappas models agree.
Let us first state Theorem \ref{thm:classicality} in the special case when $d=1$. It is reminiscent
of Gouvea's original conjecture for overconvergent modular forms (\cite{Gou},
Conjecture 3) :

\begin{thm}
a) Assume that $f$ is an overconvergent Hecke eigenform for $\mathcal{H}$,
of weight $k$, character $\chi$ for the diamond operators and with
$U_{p}$-slope not equal to $(k-2)/2$, and assume that it is not
in the image of $\theta$. Then its system of Hecke eigenvalues for
$\mathcal{H}$ comes from the $p$-stabilization of a classical form
of level $K$.

b) Assume that $f$ is an overconvergent Hecke eigenform for $\mathcal{H}$,
of weight $k$, character $\chi$ for the diamond operators and with
$U_{p}$-slope not equal to $(k-2)/2$ and less than $k-1$. Then
its system of Hecke eigenvalues for $\mathcal{H}$ comes from the
$p$-stabilization of a classical form of level $K$.
\end{thm}

We will prove the following stronger theorem,
which is a (slightly weaker) analogue of Corollary 7.2.1 of \cite{Col} (see
Remark \ref{rem: strengthening} for a strengthening of part b) ):

\begin{thm}
\label{thm:classicality, d=00003D1}a) Assume that $f$ is an overconvergent
Hecke eigenform for $\mathcal{H}$, of weight $k$, character $\chi$
for the diamond operators and assume that it is not in the image of
$\theta$. Then its system of Hecke eigenvalues for $\mathcal{H}$
is classical of level $K^{p}K_{0}(p)$.

b) Assume that $f$ is an overconvergent Hecke eigenform for $\mathcal{H}$,
of weight $k$, character $\chi$ for the diamond operators with $U_{p}$-slope
less than $k-1$. Then its system of Hecke eigenvalues is classical
of level $K^{p}K_{0}(p)$. 
\end{thm}

To do this we will aim directly at the cohomology groups $H_{rig}^{1}(X_{\overline{\mathbb{F}}_{p}}^{ord},V^{\dagger}(k-2,k-2)^{\vee})$
rather than interpreting them as summands of $H_{rig}^{1}\left(X_{\overline{\mathbb{F}}_{p}}^{ord},Sym^{k-2}\left(H_{rig}^{1}(A_{\mathbb{F}_{p}}/X_{\mathbb{F}_{p}})\right)\right)$.
The excision sequence that we are interested in is then 
\[
0\rightarrow H_{rig}^{1}(X_{\overline{\mathbb{F}}_{p}},V^{\dagger}(k-2,k-2)^{\vee})\longrightarrow H_{rig}^{1}(X_{\overline{\mathbb{F}}_{p}}^{ord},V^{\dagger}(k-2,k-2)^{\vee})\longrightarrow
\]
\[
\longrightarrow H_{X_{\overline{\mathbb{F}}_{p}}^{ss},rig}^{2}(X_{\overline{\mathbb{F}}_{p}},V^{\dagger}(k-2,k-2)^{\vee})\rightarrow H_{rig}^{2}(X_{\overline{\mathbb{F}}_{p}},V^{\dagger}(k-2,k-2)^{\vee})\rightarrow0
\]
where the first $0$ is a local $H^{1}$ which vanishes by Poincaré
duality (it corresponds to an $H^{1}$ on $X_{\overline{\mathbb{F}}_{p}}^{ss}$,
which is $0$-dimensional) and the $0$ at the end comes from the fact
that $X_{\overline{\mathbb{F}}_{p}}^{ord}$ is affine and $1$-dimensional
so any $H_{rig}^{2}$ vanishes. Rather than slopes we will analyze
this using some dimension counting analogous to parts of
\cite{Col} sections 5 and 6. The space $H_{rig}^{1}(X_{\overline{\mathbb{F}}_{p}},V^{\dagger}(k-2,k-2)^{\vee})$
looks (as a Hecke module) like two copies of the space of classical
level $K$ automorphic forms, by Matsushima's formula. The Hecke-equivariant
quotient map 
\[
H^{0}(X_{rig},W^{\dagger}(k,-k+2))\rightarrow{\rm Coker}\:\theta_{(k-2,k-2)}=H_{rig}^{1}(X_{\overline{\mathbb{F}}_{p}}^{ord},V^{\dagger}(k-2,k-2)^{\vee})
\]
 injects the space of weight $k$ level $K^{p}K_{0}(p)$ classical
$p$-new forms into $H_{rig}^{1}(X_{\overline{\mathbb{F}}_{p}}^{ord},V^{\dagger}(k-2,k-2)^{\vee})$
(this follows from Cor. \ref{cor:not in the image of theta} since
these $p$-new forms have slope $(k-2)/2$). As they are $p$-new,
they will not be in image of the map $H_{rig}^{1}(X_{\overline{\mathbb{F}}_{p}},V^{\dagger}(k-2,k-2)^{\vee})\longrightarrow H_{rig}^{1}(X_{\overline{\mathbb{F}}_{p}}^{ord},V^{\dagger}(k-2,k-2)^{\vee})$
and hence the space of weight $k$ level $K^{p}K_{0}(p)$ classical
$p$-new forms injects into $H_{X_{\overline{\mathbb{F}}_{p}}^{ss},rig}^{2}(X_{\overline{\mathbb{F}}_{p}},V^{\dagger}(k-2,k-2)^{\vee})$. 

\begin{lem}
\textup{1) Let $k\geq3$. The space of weight $k$ level $K^{p}K_{0}(p)$
classical $p$-new forms has dimension $(k-1)SS$, where $SS$ is
the number of supersingular points on $X_{\overline{\mathbb{F}}_{p}}$. }

2) \textup{The space of weight $2$ level $K^{p}K_{0}(p)$ classical
$p$-new forms has dimension $SS-1$.}\end{lem}

\begin{proof}
This is well known, we give a brief indication of the proof. 

1) In general, one shows using the Kodaira-Spencer isomorphism and
the Riemann-Roch theorem that for weight $k\geq3$ and an arbitrary
neat level $K^{\prime}$, the space of weight $k$ and level $K^{\prime}$
classical automorphic forms has dimension $(k-1)(g(X(K^{\prime}))-1)$
where $g(X(K^{\prime}))$ is the genus of the Shimura curve $X(K^{\prime})$
of level $K^{\prime}$. Let $g$ denote the genus of $X$. By looking
at $Y_{\overline{\mathbb{F}}_{p}}$, one sees that the genus of $Y$ is
$2g+SS-1$. Since the dimension of the space of weight $k$ level
$K^{p}K_{0}(p)$ classical $p$-old forms is twice that of the space
weight $k$ level $K$ classical forms (each eigenform has two $p$-stabilizations),
one gets the formula for the $p$-new forms.

2) Kodaira-Spencer shows that the space of weight $2$ and level $K^{\prime}$
classical automorphic forms has dimension $g(X(K^{\prime}))$, hence
the space of weight $2$ level $K^{p}K_{0}(p)$ classical $p$-new
forms has dimension $(2g+SS-1)-2g=SS-1$.\end{proof}

\begin{lem}
$\dim\: H_{X_{\overline{\mathbb{F}}_{p}}^{ss},rig}^{2}(X_{\overline{\mathbb{F}}_{p}},V^{\dagger}(k-2,k-2)^{\vee})=SS(k-1)$
for $k\geq2$.\end{lem}

\begin{proof}
By Poincaré duality $H_{X_{\overline{\mathbb{F}}_{p}}^{ss},rig}^{2}(X_{\overline{\mathbb{F}}_{p}},V^{\dagger}(k-2,k-2)^{\vee})=H_{rig}^{0}(X_{\overline{\mathbb{F}}_{p}}^{ss},V^{\dagger}(k-2,k-4)^{\vee})^{\vee}$
so since $X_{\overline{\mathbb{F}}_{p}}^{ss}$ is $SS$
points and $V^{\dagger}(k-2,k-4)^{\vee}$ has rank $k-1$, the formula
follows.
\end{proof}

The last ingredient of our dimension count is 

\begin{lem}
\textup{$H_{rig}^{2}(X_{\overline{\mathbb{F}}_{p}},V^{\dagger}(k-2,k-2)^{\vee})=0$
if $k\geq3$, and the one-dimensional Hecke module corresponding to
Tate twist by $-1$ if $k=2$.}\end{lem}

\begin{proof}
This follows by Matsushima's formula or other ``classical methods''
(e.g. degeneration of the BGG spectral sequence).
\end{proof}

Adding up the dimensions in the previous lemmas we see that as Hecke
modules, 
\[
H_{rig}^{1}(X_{\overline{\mathbb{F}}_{p}}^{ord},V^{\dagger}(k-2,k-2)^{\vee})=H^{0}(Y,W(k,2-k))
\]
Here we are using that image of the injection $H_{rig}^{1}(X_{\overline{\mathbb{F}}_{p}},V^{\dagger}(k-2,k-2)^{\vee})\longrightarrow H_{rig}^{1}(X_{\overline{\mathbb{F}}_{p}}^{ord},V^{\dagger}(k-2,k-2)^{\vee})$
is, as a Hecke module, the space of $p$-old forms inside $H^{0}(Y,W(k,2-k))$.
For $\mathcal{H}^{p}$ and $\left\langle p\right\rangle $ this follows
from equivariance, for $U_{p}$ this uses the same trick as in the
proof of Theorem \ref{thm:classicality}. Thus Theorem \ref{thm:classicality, d=00003D1}
follows. 
\begin{rem}
\label{rem: strengthening}The fact that these two are equal as Hecke
modules does not mean that the composition 
\[
H^{0}(Y,W(k,2-k))\hookrightarrow H^{0}(X_{rig},W^{\dagger}(k,2-k))\twoheadrightarrow coker\:\theta_{(k-2,k-2)}
\]
is an isomorphism. In the modular curve case, Coleman (\cite{Col})
shows the equality of Hecke modules as above but also that the composition
above is not an isomorphism. However, by Corollary \ref{cor:not in the image of theta},
the composition is an injection, and hence an isomorphism, on slope
$<k-1$ parts. This allows one to strengthen Theorem \ref{thm:classicality, d=00003D1}
b) to assert that $f$ itself is classical (of level $K^{p}K_{0}(p)$).
Corollary 7.2.1 of \cite{Col} asserts that the analogous strengthening
of a) is true in the case of modular curve. However, we cannot prove
it by the same technique as we do not have $q$-expansions, and as
a result do not know multiplicity $1$ for overconvergent automorphic
forms. \end{rem}


\begin{thebibliography}{}
\bibitem[1]{AnBa} Andre, Y., Baldassarri, F. De Rham Cohomology
of Differential Modules on Algebraic Varieties, Prepublication Institut
de Mathematiques de Jussieu 184.

\bibitem[2]{AnGa} Andreatta, F., Gasbarri, C. : The canonical subgroup for families of abelian varieties. Composition Math. 143(3), 566-602 (2007)

\bibitem[3]{AnGo} Andreatta, F., Goren, E. Z. : Geometry of Hilbert modular varieties over totally ramified primes. Internat. Math. Res. Not. 33, 1785-1835 (2003)

\bibitem[4]{AIP} Andreatta, F., Iovita, A., Pilloni, V. : $p$-adic
families of Siegel modular forms. Preprint, available at http://perso.ens-lyon.fr/vincent.pilloni/

\bibitem[5]{BGG} Bernstein, I. N., Gelfand, I. M., Gelfand, S.
I. : Differential operators on the base affine space and a study of
$\mathfrak{g}$-modules. In ``Lie groups and their representations'',
Summer School of the Bolyai Janos Mathematical Society, (Budapest,
1971), Adam Hilger Ltd., London (1975)

\bibitem[6]{BoWa} Borel, A., Wallach, N. Continuous cohomology,
Discrete Subgroups and Representation theory of Reductive Groups.
Second edition. Amer. Math. Soc. (1999) 

\bibitem[7]{Bou} Boutot, J. F. : Varietes de Shimura: Le probleme de modules en inegale caracteristique. In "Varietes de Shimura et fonctions L", Publ. Math. Univ. Paris VII 6 (1979)

\bibitem[8]{Buzz} Buzzard, K. M. : Eigenvarieties. In ``L-functions
and Galois representations'', 59-120, London Math. Soc. Lecture Note
Ser., 320, Cambridge Univ. Press (2007)

\bibitem[9]{BuTa} Buzzard, K. M., Taylor, R. L. : Companion forms
and weight 1 forms. Annals of Math. Vol. 149, 905-919 (1999)

\bibitem[10]{ChFa} Chai, C-L., Faltings, G. : Degenerations of
Abelian varieties. Ergebnisse der Mathematik und ihrer Grenzgebiete,
3. Folge, vol. 22, Springer-Verlag (1990)

\bibitem[11]{Che} Chenevier, G. : Familles p-adiques de formes automorphes
pour GL(n). Journal fur die reine und angewandte Mathematik 570, 143-217,
(2004)

\bibitem[12]{Col} Coleman, R. F. : Classical and overconvergent
modular forms. Inventiones Math. 124, 215-241 (1996)

\bibitem[13]{DePa} Deligne, P., Pappas, G. : Singularites des espaces de modules de Hilbert, en les caracteristiques divisant le discriminant. Compositio Math. 90, 59-79 (1994)

\bibitem[14]{Eme} Emerton, M. : On the interpolation of systems
of eigenvalues attached to automorphic Hecke eigenforms. Inventiones
Math. 164, no. 1, 1-84 (2006)

\bibitem[15]{Fal} Faltings, G. : On the Cohomology of Locally Symmetric
Hermitian Spaces. Lecture Notes in Mathematics 1029, 55-98 (1983)

\bibitem[16]{GoKa} Goren, E.Z., Kassaei, P. L. Canonical subgroups
over Hilbert modular varieties. To appear in Journal fur die reine und angewandte Mathematik. Available at http://www.mth.kcl.ac.uk/\textasciitilde{}kassaei/research/files/cshmv.pdf

\bibitem[17]{GoOo} Goren, E. Z., Oort, F. Stratifications of Hilbert
modular varieties. J. Algebraic Geometry 9, 111-154 (2000) 

\bibitem[18]{Gou} Gouvea, F. Q. Continuity properties of Modular
forms, Elliptic Curves and related topics, CRM Proceedings and Lecture
Notes, AMS 4, 85-99 (1994)

\bibitem[19]{HaTa} Harris, M., Taylor, R. L. : The Geometry and
Cohomology of Some Simple Shimura Varieties, Annals of Math. Studies
151, Princeton Univ. Press (2001)

\bibitem[20]{Hid} Hida, H. Control Theorems of coherent sheaves
on Shimura varieties of PEL type. J. Inst. of Math. Jussieu 1, 1-76 (2002)

\bibitem[21]{Hum} Humphreys, J. E. : Representations of Semisimple
Lie Algebras in the BGG Category $\mathcal{O}$. Grad. Stud. Math.,
94, Amer. Math. Soc. (2008) 

\bibitem[22]{Kas1} Kassaei, P. L. $p$-adic modular forms over
Shimura curves over $\mathbb{Q}$. PhD Thesis, Massachusetts Institute
of Technology (1999). Available at http://www.mth.kcl.ac.uk/\textasciitilde{}kassaei/research.html

\bibitem[23]{Kas2} Kassaei, P. L. : A Gluing Lemma And Overconvergent
Modular Forms. Duke Math. Journal 132 (3), 509-529 (2006)

\bibitem[24]{Ked1} Kedlaya, K. S. : Finiteness in rigid cohomology.
Duke Math. Journal 134, 15-97 (2006)

\bibitem[25]{Ked2} Kedlaya, K. S. : Fourier transforms and ``$p$-adic
Weil II''. Compositio Mathematica 142, 1426-1450 (2006)

\bibitem[26]{Kis} Kisin, M. : Overconvergent modular forms and the
Fontaine-Mazur conjecture. Inventiones Math. 153 (2), 373-454 (2003)

\bibitem[27]{KiLa} Kisin, M. Lai, K. F. Overconvergent Hilbert
Modular Forms. Amer. J. Math 127, 735-783 (2005)

\bibitem[28]{Kott} Kottwitz. Points on Shimura varieties over finite
fields. J. AMS 5, 373-444 (1992)

\bibitem[29]{Lan} Lan, K-W. : Arithmetic compactifications of PEL-type
Shimura varieties. Ph.D. thesis, Harvard University (2008)

\bibitem[30]{LS} Lan, K-W., Suh, J. : Vanishing theorems for torsion automorphic sheaves on compact PEL-type Shimura varieties. Duke Math. J. 161 , no. 6, 1113-1170 (2012)

\bibitem[31]{Lau1} Laumon, G. : Cohomology of Drinfel'd modular
varieties I. Cambridge University Press (1996)

\bibitem[32]{LaPo} Lan, K-W., Polo, P. : Dual BGG complexes for
automorphic bundles. Preprint, available at http://www.math.princeton.edu/\textasciitilde{}klan/academic.html

\bibitem[33]{LeSt1} Le Stum, B. : Rigid cohomology, volume 172
of Cambridge Tracts in Mathematics. Cambridge University Press, Cambridge
(2007)

\bibitem[34]{LeSt2} Le Stum, B. : The overconvergent site. To
appear in Memoire de la SMF (2012). Available at http://perso.univ-rennes1.fr/bernard.le-stum/Publications.html

\bibitem[35]{Loe} Loeffler, D. : Overconvergent algebraic automorphic
forms, Proc. London Math. Soc. 102, no. 2, 193-228 (2011))

\bibitem[36]{Mil79} Milne, J. S. Points on Shimura varieties mod
$p$. Proceedings of Symposia in Pure Mathematics, Vol. 33, part 2,
165-184 (1979)

\bibitem[37]{Mil} Milne, J. S. : Canonical Models of (Mixed) Shimura
varieties and Automorphic Vector Bundles. In ``Automorphic Forms,
Shimura Varieties, and L-functions'', (Proceedings of a Conference
held at the University of Michigan, Ann Arbor, July 6-16, 1988), p.
283-414. Also available at http://www.jmilne.org/math/articles/

\bibitem[38]{Mil05} Milne, J. S. Introduction to Shimura varieties.
In ``Harmonic Analysis, the Trace Formula and Shimura Varieties''
(James Arthur, Robert Kottwitz, Editors) AMS (2005) Also available
at http://www.jmilne.org/math/articles/

\bibitem[39]{MoTa} Mok, C-P., Tan, F. : Overconvergent family of
Siegel-Hilbert modular forms. Preprint, available at https://www.math.mcmaster.ca/\textasciitilde{}cpmok/

\bibitem[40]{Nak} Nakamura, K. : Classification of split trianguline representations of $p$-adic fields. Composition Mathematica, vol. 145, no. 4, 865-914 (2009)

\bibitem[41]{Pill} Pilloni, V. Prolongement analytique sur les
varietes de Siegel. Duke Math Journal, vol. 157, No. 1, 167-222 (2011)

\bibitem[42]{PS1} Pilloni, V., Stroh, B. : Surconvergence et classicite
: le cas Hilbert. Preprint, available at http://perso.ens-lyon.fr/vincent.pilloni/

\bibitem[43]{PS2} Pilloni, V., Stroh, B. : Surconvergence et classicite
: le cas deploye. Preprint, available at http://perso.ens-lyon.fr/vincent.pilloni/

\bibitem[44]{Sas} Sasaki, S. : Integral models of Hilbert modular varieties in the ramified case, deformations of modular Galois representation, and weight one forms. Preprint, available at http://www.cantabgold.net/users/s.sasaki.03/

\bibitem[45]{Shin} Shin, S. W. : On the cohomology of Rapoport-Zink
spaces of EL-type, to appear in Amer. J. Math., available at http://math.mit.edu/\textasciitilde{}swshin/

\bibitem[46]{TaYo} Taylor, R. L., Yoshida, T. : Compatibility of
local and global Langlands correspondences. J. Amer. Math. Soc., 20-2, 467-493
(2007)

\bibitem[47]{Tian} Tian, Y. : Classicality of overconvergent Hilbert
eigenforms: Case of quadratic residue degree. Preprint, available
at http://arxiv.org/abs/1104.4583

\bibitem[48]{TiXi} Tian, Y., Xiao, L. : $p$-adic cohomology and classicality of overconvergent Hilbert modular forms, available at http://math.uchicago.edu/~lxiao

\bibitem[49]{Tzu} Tzusuki, N. : On Base Change Theorem and Coherence
in Rigid Cohomology. Documenta Mathematica, Extra vol., 891-918 (2003) 

\bibitem[50]{Urb} Urban, E. : Eigenvarieties for Reductive Groups.
Annals of Mathematics, vol. 174, no. 3, 1685-1784 (2011)

\bibitem[51]{Yosh} Yoshida, T. Betti Cohomology of Shimura varieties
- the Matsushima formula. Notes, available at http://www.dpmms.cam.ac.uk/\textasciitilde{}ty245/2008\_AGR\_Fall/2008\_agr\_week2.pdf\end{thebibliography}
\end{document}